\documentclass[a4paper,12pt]{amsart}
\usepackage[T1]{fontenc}
\usepackage[centering,margin=2.3cm]{geometry}
\usepackage[latin1]{inputenc}
\usepackage{amsthm, amsmath}   
\usepackage{latexsym,amssymb}
\usepackage{algorithmic}
\usepackage{amssymb}
\usepackage{times}
\usepackage{enumerate}
\usepackage[usenames,dvipsnames]{pstricks}
\usepackage{epsfig}
\usepackage{pst-node}
\usepackage{pst-grad} 
\usepackage{pst-plot} 
\usepackage{pst-3dplot}

\newtheorem{theorem}{Theorem}
\newtheorem{lemma}{Lemma}
\newtheorem{proposition}{Proposition}

\newtheorem{example}{Example}
\newtheorem{remark}{Remark}

\newtheorem{corollary}{Corollary}

\def\dim{{\rm \ dim}\,}

\def\reg{{\rm \ reg}\, }
\def\ini{{\rm in}\, }

\def\max{{\rm max}}

\newcommand{\A}{{\mathbb A}}
\newcommand{\N}{{\mathbb N}}

\newcommand{\Z}{{\mathbb Z}}

\newcommand{\Fcal}{{\mathcal F}}

\newcommand{\Bcal}{{\mathcal B}}
\newcommand{\Acal}{{\mathcal A}}

\newcommand{\Ccal}{{\mathcal C}}

\newcommand{\Scal}{{\mathcal S}}
\newcommand{\Pcal}{{\mathcal P}}

\begin{document}
\title{Noether resolutions in dimension $2$}

\address{Universidad de La Laguna. Facultad de Ciencias.
Sección de Matemáticas.
Avda. Astrofísico Francisco Sánchez, s/n.
Apartado de correos 456.
38200-La Laguna. Tenerife. Spain.}
\author[I. Bermejo]{Isabel Bermejo}

\email{ibermejo@ull.es}

\address{Universidad de La Laguna. Facultad de Ciencias.
Sección de Matemáticas.
Avda. Astrofísico Francisco Sánchez, s/n.
Apartado de correos 456.
38200-La Laguna. Tenerife. Spain.}
\author[E. Garc\'{i}a-Llorente]{Eva Garc\'{i}a-Llorente}
\email{evgarcia@ull.es}

\address{Aix-Marseille Université, CNRS, LIF UMR 7279, Marseille, France.}
\author[I. Garc\'{i}a-Marco]{Ignacio Garc\'{i}a-Marco}
\email{ignacio.garcia-marco@lif.univ-mrs.fr, iggarcia@ull.es}

\address{Universit\'{e} de Grenoble I, Institut Fourier, UMR 5582, B.P.74, 38402 Saint-Martin D'Heres Cedex, Grenoble and ESPE de Lyon, Universit\'{e} de Lyon 1, Lyon, France.}
\author[M. Morales]{Marcel Morales}
\email{morales@ujf-grenoble.fr}

\subjclass[2010]{}
\date{}

\maketitle

\begin{abstract}
Let $R:= K[x_1,\ldots,x_{n}]$ be a polynomial ring over an infinite 
field $K$, and let $I \subset R$ be a homogeneous ideal with respect
to a weight vector $\omega = (\omega_1,\ldots,\omega_n) \in
(\Z^+)^n$ such that  $\dim(R/I) = d$. In this paper we study the minimal graded free resolution
of $R/I$ as $A$-module, that we call the Noether resolution of $R/I$, whenever $A :=K[x_{n-d+1},\ldots,x_n]$ is a Noether normalization of $R/I$. When $d=2$  and $I$ is
saturated, we give an algorithm for obtaining this resolution that involves the computation of a minimal
Gr\"obner basis of $I$ with respect to the weighted degree reverse
lexicographic order. In the particular case when $R/I$ is a
 $2$-dimensional semigroup ring, we also 
 describe the multigraded version of this resolution in terms of the underlying
semigroup.
Whenever we have the Noether resolution of $R/I$ or its multigraded version, we obtain formulas
 for the corresponding Hilbert series of $R/I$, and when $I$ is homogeneous, we obtain a formula for the Castelnuovo-Mumford regularity of $R/I$.
 Moreover, in the more general setting that $R/I$ is a simplicial 
semigroup ring of any dimension, we provide its  Macaulayfication. 

As an application of the results for $2$-dimensional semigroup rings, we 
provide a new upper bound for the Castelnuovo-Mumford regularity of the coordinate ring of a projective monomial curve.  Finally,
we describe the multigraded Noether resolution and the Macaulayfication of
 either the coordinate ring of a projective monomial curve $\Ccal\subseteq \mathbb{P}_K^{n}$  associated to an arithmetic sequence 
or the coordinate ring of any canonical projection 
 $\pi_{r}(\Ccal)$ of $\Ccal$ to $\mathbb{P}_K^{n-1}$. 
\end{abstract}


\bigskip

\section{Introduction}
Let $R:= K[x_1,\ldots,x_{n}]$ be a polynomial ring over an infinite
field $K$, and let $I \subset R$ be a weighted homogeneous ideal
with respect to the vector $\omega =
(\omega_1,\ldots,\omega_n) \in (\Z^+)^n$, i.e., $I$ is homogeneous
for the grading ${\rm deg}_{\omega}(x_i) = \omega_i$. We denote by
$d$ the Krull dimension of $R / I$ and we assume that $d\geq 1$. 
Suppose $A := K[x_{n-d+1},\ldots,x_n]$ is a Noether normalization of $R/I$, i.e.,
$A \hookrightarrow R/I$ is an integral ring extension. Under this
assumption $R/I$ is a finitely generated $A$-module, so to study the minimal graded free resolution of $R/I$ as
$A$-module is an interesting problem. Set  
$$ \mathcal F:  0 \longrightarrow \oplus_{v\in \Bcal_p} A(-s_{p,v}) \buildrel{\psi_p}\over{ \longrightarrow} \cdots \buildrel{\psi_{1}}\over{ \longrightarrow}
 \oplus_{v\in \Bcal_0} A(-s_{0,v}) \buildrel{\psi_0}\over{ \longrightarrow} R/I \longrightarrow 0$$
this resolution, where
for all $i \in \{0,\ldots,p\}$ $\Bcal_i$ denotes some finite set, and 
$s_{i,v}$ are nonnegative integers.  This
 work concerns the study of this resolution of $R/I$, which will be 
 called the {\it Noether resolution of $R/I$}. More precisely, 
 we aim at determining the sets $\Bcal_i$, the shifts $s_{i,v}$ and 
the morphisms $\psi_i$.

One of the characteristics of Noether resolutions is that they have shorter length than the minimal 
graded free resolution of $R/I$ as $R$-module. Indeed, the projective dimension of $R/I$ as $A$-module is $p = d - {\rm depth}(R/I)$, meanwhile its projective
dimension of $R/I$ as $R$-module is $n-{\rm depth}(R/I)$. Studying Noether resolutions is interesting since they contain valuable information about $R/I$. For instance,
since the Hilbert series is an additive function, we get  the Hilbert series of $R/I$ from its Noether resolution. Moreover, whenever $I$ is
a homogeneous ideal, i.e., homogeneous for the weight vector
$\omega = (1,\ldots,1)$, one can obtain the Castelnuovo-Mumford
regularity of $R/I$ in terms of the Noether resolution as ${\rm
reg}(R/I) = {\rm max}\{s_{i,v}-i \ \vert\ 0 \leq i \leq p$, $v \in
\Bcal_i\}$.

In Section 2 we start by describing in Proposition \ref{B0} the first step of the Noether resolution of $R/I$.  
By Auslander-Buchsbaum formula, the depth of $R/I$ equals $d - p$. Hence, $R/I$ is Cohen-Macaulay if and only if $p = 0$ or, equivalently, if $R/I$
is a free $A$-module.  This observation together with Proposition \ref{B0}, lead to Proposition
\ref{CMcharacterization} which is an effective
criterion for determining whether $R/I$ is Cohen-Macaulay or not. This criterion generalizes  
\cite[Proposition 2.1]{BerGi01}.  If $R/I$ is Cohen-Macaulay, Proposition \ref{B0} provides the whole Noether resolution of $R/I$.
When $d = 1$ and $R/I$ is not Cohen-Macaulay, we describe the Noether resolution of $R/I$ by means of Proposition \ref{B0} together with Proposition \ref{dim1depth0}.
Moreover, when $d = 2$ and $x_n$ is a nonzero divisor of $R/I$, we are able to provide in Theorem
\ref{algorithm_NoetherResolution} a complete description of the
Noether resolution of $R/I$. All these results rely in the computation of a minimal Gr\"obner basis of
$I$ with respect to the weighted degree reverse lexicographic order. As a consequence of this, we provide in Corollary \ref{HSm=2} a
description of the weighted Hilbert series in terms of
the same Gr\"obner basis. Whenever $I$ is a homogeneous ideal, as a consequence of Theorem
\ref{algorithm_NoetherResolution},  we obtain in Corollary
\ref{regularity} a formula for the Castelnuovo-Mumford regularity of $R/I$ which is equivalent to the one provided in
\cite[Theorem 2.7]{BerGi99}.

In section 3 we study Noether resolutions when $R/I$ is a simplicial semigroup ring,  i.e., whenever $I$ is a toric ideal and $A = K[x_{n-d+1},\ldots,x_n]$ 
is a Noether normalization of $R/I$. We recall that $I$ is a toric ideal if $I = I_{\Acal}$ with $\Acal = \{a_1,\ldots,a_n\} \subset \N^d$ and $a_i = (a_{i1},\ldots,a_{id}) \in \N^d$; where $I_{\Acal}$ denotes
the kernel of the homomorphism of $K$-algebras $\varphi: R
\rightarrow K[t_1,\ldots,t_d];\, x_i \mapsto t^{a_i} = t_1^{a_{i1}}
\cdots t_d^{a_{id}}$ for all $i \in \{1,\ldots,n\}$. If we denote by
$\Scal \subset \N^d$ the semigroup generated by $a_1,\ldots,a_n$,
then the image of $\varphi$ is $K[\Scal] := K[t^s \, \vert \, s \in
\Scal] \simeq R/I_{\Acal}$. By \cite[Corollary~4.3]{Sturm},
$I_{\Acal}$ is
multigraded with respect to the grading induced by $\Scal$ which assigns ${\rm deg}_{\Scal}(x_i) = a_i$ for all 
$i \in \{1,\ldots,n\}$. Moreover, whenever $A$ is a Noether normalization of $K[\Scal]$ we may assume without loss of generality that  $a_{n-d+i} = w_{n-d+i} e_i$ for all $i \in \{1,\ldots,d\}$, where $\omega_{n-d+i}\in\Z^+$ and $\{e_1,\ldots,e_d\}$ is the canonical basis of $\N^d$.  In this setting  we may consider a {\it multigraded Noether
resolution} of $K[\Scal]$, i.e., a minimal multigraded free
resolution of $K[\Scal]$ as $A$-module:
$$0 \longrightarrow  \oplus_{s \in \Scal_p} A \cdot
s  \buildrel{\psi_p}\over{ \longrightarrow} \cdots
\buildrel{\psi_1}\over{ \longrightarrow}  \oplus_{s \in \Scal_0} A
\cdot s \buildrel{\psi_0}\over{ \longrightarrow} K[\Scal]
\longrightarrow 0,$$ 
where $\Scal_i$ are finite subsets of $\Scal$ for all $i \in
\{0,\ldots,p\}$ and $A \cdot s$ denotes the shifting of $A$ by $s
\in \Scal$.  We observe that this multigrading is a refinement
of the grading given by $\omega = (\omega_1,\ldots,\omega_n)$ with
$\omega_i := \sum_{j=1}^d a_{ij} \in \Z^+$; thus, $I_{\Acal}$ is
weighted homogeneous with respect to $\omega$. As a consequence,  whenever we get the multigraded Noether resolution or the multigraded Hilbert series 
of $K[\Scal]$, we also obtain
its Noether resolution and its Hilbert series with respect to the weight vector $\omega$.

 A natural and interesting problem is to describe combinatorially the multigraded Noether resolution of $K[\Scal]$ in terms of the semigroup $\Scal$.
This approach would lead us to results for simplicial semigroup
rings $K[\Scal]$ which do not depend on the characteristic of the field $K$. In general, for any toric ideal, it is well known
 that the minimal number of binomial generators of
$I_{\Acal}$ does not depend on the characteristic of $K$ (see, e.g., \cite[Theorem
5.3]{Sturm}),  but the
Gorenstein,  Cohen-Macaulay and 
Buchsbaum properties of $K[\Scal]$ depend on the characteristic of $K$ (see \cite{Hoa91}, \cite{TH} and \cite{Hoa88}, respectively).
However, in the context of simplicial semigroup rings, these properties do not depend on the characteristic of $K$ (see 
\cite{GotoSuzukiWatanabe76}, \cite{Stanley} and \cite{GarciaSanchezRosales02}, respectively). These facts give support to our aim
of describing the whole multigraded Noether resolution of $K[\Scal]$ in terms of the underlying semigroup $\Scal$ for simplicial semigroup rings.

The results in section 3 are the following. In Proposition \ref{S0} we describe the first step of the multigraded Noether resolution of a simplicial semigroup ring $K[\Scal]$.  As a byproduct we recover in Proposition
\ref{CMcharacSemigroup} a well-known criterion for $K[\Scal]$ to be
Cohen-Macaulay in terms of the semigroup. When $d = 2$, i.e., $I_{\Acal}$ is the ideal of an affine toric
surface, Theorem \ref{S_1} describes the second step of the multigraded Noether resolution in terms of the semigroup $\Scal$. When $d = 2$, from Proposition \ref{S0} and Theorem \ref{S_1}, we derive the whole multigraded Noether resolution of $K[\Scal]$ by means of
 $\Scal$ and, as a byproduct, we also get in Corollary \ref{multiHilbert} its multigraded Hilbert series. 
 Whenever $I_{\Acal}$ a is homogeneous ideal, we get a formula for the Castelnuovo-Mumford regularity of $K[\Scal]$ in terms of $\Scal$, see Remark \ref{multtostandardgraded}.

Given an algebraic variety, the set of points where $X$ is not Cohen-Macaulay is the non Cohen-Macaulay locus.  
Macaulayfication is an analogous operation to resolution of singularities and was
considered in Kawasaki \cite{Kawasaki}, where he provides certain sufficient conditions for $X$ to admit a Macaulayfication. 
For semigroup rings Goto et al. \cite{GotoSuzukiWatanabe76} and Trung and Hoa
\cite{TH} proved the existence of a semigroup $\Scal'$ satisfying
$\Scal \subset \Scal' \subset \bar{\Scal}$, where $\bar{\Scal}$ denotes the saturation of $\Scal$ and thus $K[\bar{\Scal}]$ is the
normalization of $K[\Scal]$, such that we have an
exact sequence:

$$0 \longrightarrow K[\Scal] \longrightarrow K[\Scal'] \longrightarrow K[\Scal' \setminus \Scal] \longrightarrow 0$$
with ${\rm dim}(K[\Scal' \setminus \Scal]) \leq {\rm dim}(K[\Scal]) -
2$. In this setting, $K[\Scal']$ satisfies the condition $S_2$ of Serre, and  is called the $S_2$-fication of
$K[\Scal]$.  Moreover, when $\Scal$ is a simplicial semigroup, \cite[Theorem 5]{Morales07} proves that this semigroup ring $K[\Scal']$ is
exactly the Macaulayfication of $K[\Scal]$; indeed, he proved that  
 $K[\Scal']$ is Cohen-Macaulay and the support of $K[\Scal' \setminus \Scal]$ 
 coincides with the non Cohen-Macaulay locus of $K[\Scal]$.  
 In \cite{Morales07}, the author provides an explicit description of
the Macaulayfication of $K[\Scal]$ in terms of the system of
generators of $I_{\Acal}$ provided $K[\Scal]$ is a codimension $2$
simplicial semigroup ring. Section 4 is devoted to study the Macaulayfication of any simplicial semigroup ring. The main result of this section is  Theorem
\ref{macaulayfication_semigroups}, where we entirely describe the Macaulayfication of
any simplicial semigroup ring $K[\Scal]$ in terms of the set $\Scal_0$, the subset of $\Scal$ that provides the first step of the multigraded Noether 
resolution of $K[\Scal]$.

In sections 5 and 6 we apply the methods and results obtained in the previous ones to
certain dimension $2$ semigroup rings. More precisely, a sequence  $m_1 < \cdots < m_n$
determines the projective monomial curve
$\Ccal \subset {\mathbb P}_K^n$ parametrically defined by $x_i :=
s^{m_i} t^{m_n-m_i}$ for all $i \in \{1,\ldots,n-1\},\, x_{n} =
s^{m_n},  \, x_{n+1} := t^{m_n}$. If we set $\Acal =
\{a_1,\ldots,a_{n+1}\} \subset \N^2$ where $a_i := (m_i, m_n - m_i),
a_n := (m_n, 0)$ and $a_{n+1} := (0,m_n)$, it turns out that the
homogeneous coordinate ring of $\Ccal$ is $K[\Ccal] := K[x_1,\ldots,x_{n+1}]/I_{\Acal}$ and $A = K[x_n,x_{n+1}]$ is a Noether normalization of $R/I_{\Acal}$.

 The main result in Section 5 is Theorem \ref{upperboundregmoncurve}, where
we provide an upper bound on the Castelnuovo-Mumford regularity of $K[\Ccal]$, where $\Ccal$ is a projective 
monomial curve. The proof of this bound is elementary and builds on the results of the previous sections together with
some classical results on numerical semigroups. 
It is known that ${\rm reg}(K[\Ccal]) \leq m_n - n + 1$ after the work \cite{GLP}.
In our case, \cite{Lvovsky} obtained a better upper bound, indeed if we set $m_0 := 0$ he proved that 
${\rm reg}(K[\Ccal]) \leq {\rm max}_{1 \leq i < j \leq n}\{m_i - m_{i-1} + m_j - m_{j-1}\} - 1$. The proof provided by L'vovsky is quite involved and uses advanced cohomological
tools, it would be interesting to know if our results could yield a combinatorial alternative proof of this result. Even if L'vovsky's bound usually
gives a better estimate than the bound we provide here, we easily construct families such that our bound outperforms the one by L'vovsky.

 Also in the context of projective monomial curves, whenever $m_1 < \cdots < m_n$ is an arithmetic sequence of
relatively prime integers, the simplicial semigroup ring $R/I_{\Acal}$ has been
extensively studied (see, e.g., \cite{MolinelliTamone1995,
LiPatilRoberts12, BerGarGar15}) and the multigraded Noether resolution is easy to obtain.
In Section 6, we study the coordinate ring
of the canonical projections of projective monomial curves associated to arithmetic
sequences, i.e., the curves $\Ccal_r$ whose homogeneous coordinate
rings are $K[\Scal_r] = R/I_{\Acal_r}$, where $\Acal_r := \Acal \setminus \{a_r\}$
and $\Scal_r \subset \N^2$ the semigroup generated by $\Acal_r$ for all $r \in \{1,\ldots,n-1\}$. In Corollary \ref{CMcharacterization_QuasiArithmetics} we give a criterion for
 determining when the semigroup ring $K[\Scal_r]$ is Cohen-Macaulay; whenever it is not Cohen-Macaulay, we get
 its Macaulayfication in Corollary \ref{macaulayfication}. Furthermore, in Theorem
\ref{resolution_QuasiArithmetics} we provide an explicit description
of their multigraded Noether resolutions. Finally, in 
Theorem \ref{regularity_QuasiArithmetics} we get a formula for their
Castelnuovo-Mumford regularity.

\section{Noether resolution. General case}

Let $R:=K[x_1,\ldots,x_n]$ be a polynomial ring over an infinite
field $K$, and let $I \subset R$ be a $\omega$-homogeneous ideal, i.e., a weighted homogeneous ideal with respect to the vector 
$\omega = (\omega_1,\ldots,\omega_n) \in (\Z^+)^n$. We assume that $A:=K[x_{n-d+1},\ldots,x_n]$
is a Noether normalization of $R/I$, where $d := {\rm dim}(R/I)$. In
this section we study the Noether resolution of $R/I$, i.e., the minimal graded free resolution of $R/I$ as $A$-module:
\begin{equation}\label{GenResolution_m} \mathcal F:  0 \longrightarrow \oplus_{v\in \Bcal_p} A(-s_{p,v}) \buildrel{\psi_p}\over{ \longrightarrow} \cdots \buildrel{\psi_{1}}\over{ \longrightarrow}
 \oplus_{v\in \Bcal_0} A(-s_{0,v}) \buildrel{\psi_0}\over{ \longrightarrow} R/I \longrightarrow 0,\end{equation}
\noindent where
for all $i \in \{0,\ldots,p\}$ $\Bcal_i$ is a finite set of monomials, and 
$s_{i,v}$ are nonnegative integers.

In order to obtain the first step of the resolution, we will deal with the initial ideal of $I + (x_{n-d+1},\ldots,x_n)$ 
with respect to the weighted degree reverse lexicographic order  $>_{\omega}$. 

We recall that $>_{\omega}$ is defined as follows: $x^{\alpha} >_{\omega} x^{\beta}$
if and only if
\begin{itemize}
\item $\deg_{\omega}(x^{\alpha})>\deg_{\omega}(x^{\beta}),$  or
\item $\deg_{\omega}(x^{\alpha})=\deg_{\omega}(x^{\beta})$ and  the
 last  nonzero  entry of $\alpha - \beta \in \Z^n$ is
 negative. \end{itemize} 

For every polynomial $f\in R$ we denote by $\ini(f)$ 
the initial term of $f$ with respect to $>_{\omega}$. Analogously, for
every ideal $J \subset R$,  $\ini(J)$ denotes its initial ideal with respect to $>_{\omega}$.

\begin{proposition}\label{B0}
Let $\Bcal_0$ be the set of monomials that do not belong to $\ini(I+(x_{n-d+1},\ldots,x_n))$ Then,
$$\{x^{\alpha} + I\ \vert\ x^{\alpha}\in\Bcal_0\}$$ is a minimal
set of generators of $R/I$ as $A$-module and the shifts of the first step of
the Noether resolution (\ref{GenResolution_m}) are given by ${\rm
deg}_{\omega}(x^{\alpha})$ with $x^{\alpha} \in \Bcal_0$.
\end{proposition}
\begin{proof} Since $A$ is a Noether
normalization of $R/I$ we have that $\Bcal_0$ is a 
finite set. Let $\Bcal_0 = \{x^{\alpha_1},\ldots,x^{\alpha_k}\}$. To prove that $\Bcal := \{x^{\alpha_1} + I,\ldots,x^{\alpha_k}+I\}$ is a set of generators of $R/I$ as $A$-module 
it suffices to show that for every monomial $x^{\beta} := x_1^{\beta_1} \cdots x_{n-d}^{\beta_{n-d}} \notin \ini(I)$, one has that $x^{\beta} + I \in R/I$ can be written
as a linear combination of $\{x^{\alpha_1} + I,\ldots, x^{\alpha_k} + I\}$.  Since $\{x^{\alpha_1} + (I + (x_{n-d+1},\ldots,x_n)),\ldots,x^{\alpha_k} + (I+(x_{n-d+1},\ldots,x_n))\}$
is a
$K$-basis of $R/(I + (x_{n-d+1},\ldots,x_n))$, we have that $g := x^{\beta} -  \sum_{i = 1}^k \lambda_i x^{\alpha_i} \in I+(x_{n-d+1},\ldots,x_n)$ for some $\lambda_1,\ldots,\lambda_k \in K$. Then we deduce that $\ini(g) \in \ini(I+(x_{n-d+1},\ldots,x_n))$ which is equal to $\ini(I) + (x_{n-d+1},\ldots,x_n)$, and thus
$\ini(g) \in \ini(I)$. Since $x^{\beta} \notin \ini(I)$ and $x^{\alpha_i} \notin \ini(I)$ for all $i \in \{1,\ldots,k\}$, we conclude that $g = 0$ and $x^{\beta} + I =  (\sum_{i = 1}^k \lambda_i x^{\alpha_i}) + I$. The minimality of $\Bcal$ can be easily proved. \end{proof}

\medskip

When $R/I$ is a free $A$-module or, equivalently, when the
projective dimension of $R/I$ as $A$-module is $0$ and hence $R/I$ is
Cohen-Macaulay, Proposition
\ref{B0} provides the whole Noether resolution of $R/I$. In Proposition \ref{CMcharacterization} we characterize the
Cohen-Macaulay property for $R/I$ in terms of the initial ideal $\ini(I)$ previously defined.
This result generalizes
\cite[Theorem 2.1]{BerGi01}, which applies for $I$ a homogeneous
ideal.

\begin{proposition}\label{CMcharacterization}
Let $A = K[x_{n-d+1},\ldots,x_n]$ be a Noether normalization of $R/I$. Then,
$R/I$ is Cohen-Macaulay if and only if $x_{n-d+1},\ldots,x_{n}$ do
not divide any minimal generator of $\ini(I)$.
\end{proposition}
\begin{proof}
We denote by $\{e_v \, \vert \, v$ in $\Bcal_0\}$ the canonical basis
of $\oplus_{v\in \Bcal_0} A(-{\rm deg}_{\omega}(v))$. By  Proposition
\ref{B0} we know that $\psi_0: \oplus_{v\in \Bcal_0} A(-{\rm
deg}_{\omega}(v)) \longrightarrow R/I$ is the morphism induced by $e_v \mapsto v + I \in
R/I$. By Auslander-Buchsbaum formula, $R/I$ is Cohen-Macaulay if and only if $\psi_0$ is
injective.

 $(\Rightarrow)$ By contradiction, we assume that there exists
 $\alpha = (\alpha_1,\ldots, \alpha_n) \in \N^n$ such that
$x^{\alpha} = x_1^{\alpha_1} \cdots x_{n}^{\alpha_{n}}$ is a minimal
generator of $\ini(I)$ and that $\alpha_i > 0$ for some $i \in
\{n-d+1,\ldots,n\}$. Set $u := x_1^{\alpha_1}\cdots
x_{n-d}^{\alpha_{n-d}}$, since ${\rm in}(I+(x_{n-d+1},\ldots,x_n)) = {\rm in}(I) + (x_{n-d+1},\ldots,x_n)$, we have that $ u\in \Bcal_0$. We
also set 
$x^{\alpha'} := x_{n-d+1}^{\alpha_{n-d+1}}\cdots x_n^{\alpha_n} 
\in A$ and  $f$ the remainder of $x^{\alpha}$ modulo the reduced Gr\"obner basis of $I$
with respect to $>_{\omega}$. Then $x^{\alpha} - f \in I$ and every
monomial in $f$ does not belong to $\ini(I)$. As a consequence, $f =
\sum_{i = 1}^t c_{i} x^{\beta_i}$, where $c_i \in K$ and
$x^{\beta_i} = v_{i} x^{\beta_i'}$ with $v_{i} \in \Bcal_0$ and
$x^{\beta_i'} \in A$ for all $i \in \{1,\ldots,t\}$. Hence,
$x^{\alpha'} e_u - \sum_{i = 1}^t c_{i} x^{\beta_i'} e_{v_i}
\in {\rm Ker}(\psi_0)$ and $R/I$ is not Cohen-Macaulay.

 $(\Leftarrow)$ Assume that there exists a nonzero $g \in {\rm
Ker}(\psi_0)$, namely, $g = \sum_{v\in\Bcal_0} g_v e_v \in {\rm
Ker}(\psi_0)$ with $g_v \in A$ for all $v \in \Bcal_0$. Then,
$\sum_{v\in\Bcal_0} g_v v\in I$. We write $\ini(g) = c x^{\alpha} u$
with $c \in K$, $x^{\alpha} \in A$ and $u \in \Bcal_0$. Since
$x_{n-d+1},\ldots,x_n$ do not divide any minimal generator of
$\ini(I)$, we have that $u \in \ini(I)$, a contradiction.
\end{proof}

\medskip

When $R/I$ has dimension $1$, its depth can be either $0$ or $1$. When ${\rm depth}(R/I) = 1$, then $R/I$ is Cohen-Macaulay and the whole
Noether resolution is given by Proposition \ref{B0}. When $R/I$ is not Cohen-Macaulay, then its depth is $0$ and its projective dimension as
$A$-module is $1$. In this setting, to describe the whole Noether resolution it remains to determine $\Bcal_1$, $\psi_1$ and the
shifts $s_{1,v} \in \N$ for all $v \in \Bcal_1$. In Proposition \ref{dim1depth0} we
 explain how to obtain $\Bcal_1$ and $\psi_1$ by means of a Gr\"obner basis of $I$ with respect
to $>_{\omega}$.

Consider $\chi_1: R \longrightarrow R$ the evaluation morphism  induced by $x_i \mapsto x_i$ for $i \in
\{1,\ldots,n-1\}$, $x_n \mapsto 1$. 

\bigskip

\begin{proposition}\label{dim1depth0}Let $R/I$ be $1$-dimensional ring of depth $0$.
Let $L$ be the ideal $\chi_1({\rm in}(I)) \cdot R$. Then,
$$\Bcal_1 = \Bcal_0 \cap L$$
in the Noether resolution (\ref{GenResolution_m}) of $R/I$ and the shifts of the second step of this resolution
are given by ${\rm deg}_{\omega}(u x_{n}^{\delta_u}),$ where $u \in \Bcal_1$ and $\delta_u := {\rm min}
\{\delta \, \vert \, u x_{n}^{\delta} \in {\rm in}(I)\}$.
\end{proposition}
\begin{proof} For every $u =
x_1^{\alpha_1} \cdots x_{n-1}^{\alpha_{n-1}} \in \Bcal_0 \cap L$,
there exists $\delta \in \N$ such that $u x_{n}^{\delta} \in {\rm
in}(I)$; let $\delta_u$ be the minimum of all such $\delta$.
 Consider $p_u \in R$ the remainder of $u x_{n}^{\delta_u}$ modulo the reduced Gr\"obner
basis of $I$ with respect to $>_{\omega}$. Thus $u
x_{n}^{\delta_u} - p_u \in I$ is $\omega$-homogeneous and every
monomial $x^{\beta}$ appearing in $p_u$ does not belong to ${\rm
in}(I)$, then by Proposition \ref{B0} it can be expressed as
$x^{\beta} = v x_{n}^{\beta_{n}}$, where
$\beta_{n} \geq 0$ and $v \in \Bcal_0$. Moreover, since $ux_{n}^{\delta_u} >_{\omega} x^{\beta}$, then $\beta_{n} \geq
\delta_u$ and $u >_{\omega} v$. Thus, we can write
$$p_u = \sum_{v \in \Bcal_0 \atop u >_{\omega} v} x_{n}^{\delta_u}
m_{u_v} v,$$ with $m_{u_v} = c x^{\alpha_{u_v}} \in
A = K[x_n]$ a monomial (possibly $0$) for all $v \in \Bcal_0$, $u
>_{\omega} v$.

Now we denote by $\{e_v \, \vert \, v$ in $\Bcal_0\}$ the canonical
basis of $\oplus_{v\in \Bcal_0} A(-\deg_{\omega}(v))$ and consider
the graded morphism $\psi_0: \oplus_{v\in \Bcal_0}
A(-\deg_{\omega}(v)) \longrightarrow R/I$ induced by $e_v \mapsto v + I
\in R/I$. The above construction yields that $$h_u :=
x_{n}^{\delta_u} (e_u - \sum_{v \in \Bcal_0 \atop u >_{\omega} v}
m_{u_v} e_v) \in {\rm Ker}(\psi_0)$$ for all $u \in \Bcal_0 \cap L$. We will 
prove that ${\rm Ker}(\psi_0)$ is a free $A$-module with basis
$$\Ccal := \{h_u \, \vert \, u \in \Bcal_0 \cap L\}.$$

Firstly, we observe that the $A$-module generated by the elements of
$\Ccal$ is free due to the triangular form of the matrix formed by the elements of $\Ccal$. 
Let us now take $g = \sum_{v \in \Bcal_0} g_v e_v \in {\rm
Ker}(\psi_0)$ with $g_v \in A$, we assume that $g \in \oplus_{v
\in \Bcal_0} A(-\deg_{\omega}(v))$ is $\omega$-homogeneous and, thus, $g_v$ is
either $0$ or a monomial of the form $c x_n^{\beta_v}$ with $c \in K$ and $\beta_v \in \N$ for all $v \in
\Bcal_0$. We consider $\bar{\psi_0}:
\oplus_{v \in \Bcal_0} A(-\deg_{\omega}(v)) \longrightarrow R$ the
monomorphism of $A$-modules induced by $e_v \mapsto v$. Since
$\psi_0(g) = 0$, then the polynomial $g' := \bar{\psi_0}(g) =
\sum_{u \in \Bcal_0} g_u u \in I$ and ${\rm in}(g') = c
x_{n}^{\gamma} w$ for some $w \in \Bcal_0$, $\gamma \in \N$ and $c \in K$. Since ${\rm in}(g') \in {\rm in}(I)$, we 
get that $w \in \Bcal_0 \cap L$ and $\gamma \geq \delta_w$. Hence, $g_1 := g - c x_{n-1}^{\gamma
- \delta_w} h_w \in {\rm Ker}(\psi_0)$. If
$g_1$ is identically zero, then $g \in (\{h_u \, \vert \, u \in
\Bcal_0 \cap L\})$. If $g_1$ is not zero, we have that $0 \neq {\rm
in}(\bar{\psi_0}(g_1)) < {\rm in}(\bar{\psi_0}(g))$ and we iterate this process with
$g_1$ to derive that $\{h_u \, \vert \, u \in \Bcal_0 \cap L\}$
generates ${\rm Ker}(\psi_0)$.
\end{proof}

The rest of this section concerns $I$ a saturated ideal  such that
$R/I$ is $2$-dimensional and it is not Cohen-Macaulay  (and, in particular, ${\rm depth}(R/I) = 1$). We assume that $A = K[x_{n-1},x_n]$ is a Noether normalization of
$R/I$ and we aim at  describing the whole Noether resolution of $R/I$. To
achieve this it only remains to describe $\Bcal_1$, $\psi_1$ and the
shifts $s_{1,v} \in \N$ for all $v \in \Bcal_1$. In Proposition \ref{Cmonomials} we
 explain how to obtain $\Bcal_1$ and $\psi_1$ by means of a Gr\"obner basis of $I$ with respect
to $>_{\omega}$. Since $K$ is an infinite field, $I$ is a saturated ideal and $A$ is a Noether
normalization of $R/I$, one has that $x_n + \tau x_{n-1}$ is 
a nonzero divisor on $R/I$ for all $\tau \in K$ but a finite set.  Thus, by performing a mild change of
coordinates if necessary, we may assume that $x_n$ is a nonzero divisor on $R/I$.

Now consider $\chi: R \longrightarrow R$ the evaluation morphism  induced by $x_i \mapsto x_i$ for $i \in
\{1,\ldots,n-2\}$, $x_i \mapsto 1$ for $i \in \{n-1,n\}$. 

\bigskip

\begin{proposition}\label{Cmonomials}  Let $R/I$ be $2$-dimensional, non Cohen-Macaulay ring such that $x_n$ is a nonzero divisor.
Let $J$ be the ideal $\chi({\rm in}(I)) \cdot R$. Then,
$$\Bcal_1 = \Bcal_0 \cap J$$
in the Noether resolution (\ref{GenResolution_m}) of $R/I$ and the shifts of the second step of this resolution
are given by ${\rm deg}_{\omega}(u x_{n-1}^{\delta_u}),$ where $u \in \Bcal_1$ and $\delta_u := {\rm min}
\{\delta \, \vert \, u x_{n-1}^{\delta} \in {\rm in}(I)\}$.
\end{proposition}
\begin{proof} Since $x_n$ is a nonzero divisor of $R/I$ and $I$ is
a $\omega$-homogeneous ideal, then $x_n$ does not divide any minimal
generator of $\ini(I)$. As a consequence, for every $u =
x_1^{\alpha_1} \cdots x_{n-2}^{\alpha_{n-2}} \in \Bcal_0 \cap J$,
there exists $\delta \in \N$ such that $u x_{n-1}^{\delta} \in {\rm
in}(I)$; by definition, $\delta_u$ is the minimum of all such $\delta $.
 Consider $p_u \in R$ the remainder of $u x_{n-1}^{\delta_u}$ modulo the reduced Gr\"obner
basis of $I$ with respect to $>_{\omega}$. Then $u
x_{n-1}^{\delta_u} - p_u \in I$ is $\omega$-homogeneous and every
monomial $x^{\beta}$ appearing in $p_u$ does not belong to ${\rm
in}(I)$, then by Proposition \ref{B0} it can be expressed as
$x^{\beta} = v x_{n-1}^{\beta_{n-1}} x_{n}^{\beta_{n}}$, where
$\beta_{n-1}, \beta_{n} \geq 0$ and $v \in \Bcal_0$. Moreover, we have
that $ux_{n-1}^{\delta_u} >_{\omega} x^{\beta}$  which implies that
either $\beta_{n} \geq 1$, or $\beta_{n}=0$, $\beta_{n-1} \geq
\delta_u$ and $u >_{\omega} v$. Thus, we can write
$$p_u = \sum_{v \in \Bcal_0 \atop u >_{\omega} v} x_{n-1}^{\delta_u}
f_{u_v} v + \sum_{v \in \Bcal_0} x_{n} g_{u_v} v,$$ with $f_{u_v} \in
K[x_{n-1}]$ for all $v \in \Bcal_0$, $u
>_{\omega} v$ and $g_{u_v} \in A$ for all $v \in \Bcal_0$.

Now we denote by $\{e_v \, \vert \, v$ in $\Bcal_0\}$ the canonical
basis of $\oplus_{v\in \Bcal_0} A(-\deg_{\omega}(v))$ and consider
the graded morphism $\psi_0: \oplus_{v\in \Bcal_0}
A(-\deg_{\omega}(v)) \longrightarrow R/I$ induced by $e_v \mapsto v + I
\in R/I$. The above construction yields that $$h_u :=
x_{n-1}^{\delta_u} e_u - \sum_{v \in \Bcal_0 \atop u >_{\omega} v}
x_{n-1}^{\delta_u} f_{u_v} e_v - \sum_{v \in \Bcal_0} x_{n} g_{u_v}
e_v \in {\rm Ker}(\psi_0)$$ for all $u \in \Bcal_0 \cap J$. We will 
prove that ${\rm Ker}(\psi_0)$ is a free $A$-module with basis
$$\Ccal := \{h_u \, \vert \, u \in \Bcal_0 \cap J\}.$$

Firstly, we prove that the $A$-module generated by the elements of
$\Ccal$ is free. Assume that $\sum_{u \in \Bcal_0 \cap J} q_u h_u = 0$
where $q_u \in A$ for all $u \in \Bcal_0 \cap J$ and we may also
assume that $x_{n}$ does not divide $q_v$ for some $v \in \Bcal_0 \cap
J$. We consider the evaluation morphism $\tau$ induced by $x_{n}
\mapsto 0$ and we get that $\sum_{u \in \Bcal_0 \cap J} \tau(q_u)\,
\tau(h_u) = \sum_{u \in \Bcal_0 \cap J} \tau(q_u)\,
(x_{n-1}^{\delta_u} e_u + \sum_{v \in \Bcal_0 \atop u >_{\omega} v}
x_{n-1}^{\delta_u} f_{u_v} e_v) = 0$, which implies that $\tau(q_u)
= 0$ for all $u \in \Bcal_0 \cap J$ and, hence, $x_{n} \mid q_u$ for
all $u \in \Bcal_0 \cap J$, a contradiction.

Let us take $g = \sum_{v \in \Bcal_0} g_v e_v \in {\rm
Ker}(\psi_0)$ with $g_v \in A$, we assume that $g \in \oplus_{v
\in \Bcal_0} A(-\deg_{\omega}(v))$ is $\omega$-homogeneous and, thus, $g_v$ is
either $0$ or a $\omega$-homogeneous polynomial for all $v \in
\Bcal_0$. We may also suppose that there exists $v \in \Bcal_0$ such
that $x_{n}$ does not divide $g_v$. We consider $\bar{\psi_0}:
\oplus_{v \in \Bcal_0} A(-\deg_{\omega}(v)) \longrightarrow R$ the
monomorphism of $A$-modules induced by $e_v \mapsto v$. Since
$\psi_0(g) = 0$, then the polynomial $g' := \bar{\psi_0}(g) =
\sum_{u \in \Bcal_0} g_u u \in I$ and ${\rm in}(g') = c
x_{n-1}^{\gamma} w$ for some $w \in \Bcal_0$ and some $c \in K$, which
implies that $w \in \Bcal_0 \cap J$. By definition of $\delta_w$ we
get that $\gamma \geq \delta_w$, hence $g_1 := g - c x_{n-1}^{\gamma
- \delta_w} h_w \in {\rm Ker}(\psi_0)$. If
$g_1$ is identically zero, then $g \in (\{h_u \, \vert \, u \in
\Bcal_0 \cap J\})$. If $g_1$ is not zero, we have that $0 \neq {\rm
in}(\bar{\psi_0}(g_1)) < {\rm in}(\bar{\psi_0}(g))$ and we iterate this process with
$g_1$ to derive that $\{h_u \, \vert \, u \in \Bcal_0 \cap J\}$
generates ${\rm Ker}(\psi_0)$.

\end{proof}

\bigskip

From Propositions \ref{B0} and \ref{Cmonomials} and their
proofs, we can obtain the Noether resolution $\mathcal F$ of $R/I$
by means of a Gr\"obner basis of $I$ with respect to $>_{\omega}$.
We also observe that for obtaining the shifts of the resolution it
suffices to know a set of generators of $\ini(I)$. The following theorem gives the resolution. 

\begin{theorem}\label{algorithm_NoetherResolution}  Let $R/I$ be a $2$-dimensional ring such that $x_n$ is a nonzero divisor. We denote by $\mathcal{G}$ be a Gr\"obner basis of
$I$ with respect to $>_{\omega}$. If $\delta_u:={\rm min}\{\delta\
\vert\ u x_{n-1}^{\delta} \in \ini(I)\}$ for all $u\in\Bcal_1$,
then
$$\mathcal F: 0\longrightarrow  \oplus_{u \in \Bcal_1} A(-\deg_{\omega}(u)-\delta_u \omega_{n-1})
\buildrel{\psi_1}\over{ \longrightarrow}   \oplus_{v\in\Bcal_0}
A(-\deg_{\omega}(v)) \buildrel{\psi_0}\over{\longrightarrow}
R/I\longrightarrow 0,$$ is the Noether resolution of $R/I$, where
$$\begin{array}{crcl} \psi_0: & \oplus_{v \in  \Bcal_0}
A(-\deg_{\omega}(v)) & \rightarrow & R/I,
\\ & e_v & \mapsto &v + I \end{array}$$ and $$\begin{array}{crcl} \psi_1: & \oplus_{u \in \Bcal_1} A(-\deg_{\omega}(u)-\delta_u \omega_{n-1}) & \longrightarrow
& \oplus_{v \in \Bcal_0} A(-\deg_{\omega}(v)), \\ & e_u & \mapsto &
x_{n-1}^{\delta_u} e_u - \sum_{v\in\Bcal_0}
 f_{u_v} e_v \end{array}$$ whenever $\sum_{v \in \Bcal_0} f_{u_v} v$ with $f_{u_v} \in A$ is the remainder of the division
 of $u x_{n-1}^{\delta_u}$ by $\mathcal G$.
\end{theorem}

\bigskip

From this resolution, we can easily describe the weighted
Hilbert series of $R/I$.

\begin{corollary}\label{HSm=2} Let $R/I$ be a $2$-dimensional ring such that $x_n$ is a nonzero divisor, then its Hilbert series is given by:
$$HS_{R/I}(t)= \frac{\sum_{v\in\Bcal_0}t^{\deg_{\omega}(v)} - \sum_{u \in\Bcal_1}t^{\deg_{\omega}(u)+\delta_u w_{n-1}}}{(1-t^{\omega_{n-1}})(1-t^{\omega_{n}})}$$
\end{corollary}

\noindent In the following example we show how to compute the Noether resolution and the weighted Hilbert series of 
the graded coordinate ring of a surface in $\A_K^4$.

\begin{example}\label{ejemplo_m2}Let $I$ be the defining ideal of the surface of $\A_K^4$ 
parametrically defined by  $f_1:= s^3+s^2t,f_2:= t^4+st^3,f_3:=s^2, f_4:=t^2
\in K[s,t]$. Using {\sc Singular} \cite{DGPS}, {\sc CoCoA}
\cite{cocoa5} or  {\sc Macaulay 2}  \cite{macaulay} we obtain that whenever
$K$ is a characteristic $0$ field, the polynomials
$\{g_1,g_2,g_3,g_4\}$ constitute a minimal Gr\"obner basis of its
defining ideal with respect to $>_{\omega}$ with $\omega =
(3,4,2,2)$, where $g_1 := 2x_2x_3^2-x_1^2x_4+x_3^3x_4-x_3^2x_4^2, \
g_2 := x_1^4-2x_1^2x_3^3+x_3^6-2x_1^2x_3^2x_4-2x_3^5x_4+x_3^4x_4^2,
\ g_3 := x_2^2-2x_2x_4^2-x_3x_4^3+x_4^4$ and $g_4 :=
2x_1^2x_2-x_1^2x_3x_4+x_3^4x_4-3x_1^2x_4^2-2x_3^3x_4^2+x_3^2x_4^3.$
In particular, $$\ini(I) = (x_2x_3^2, x_1^4, x_2^2, x_1^2 x_2).$$
Then, we obtain that 
\begin{itemize} \item $\Bcal_0 =
\{u_1,\ldots,u_6\}$ with $u_1:=1,\, u_2:=x_1,\,u_3:=x_2,\,
u_4:=x_1^2,\,u_5:=x_1x_2,\,u_6:=x_1^3,$ \item $J=(x_2,x_1^4) \subset K[x_1,x_2,x_3,x_4]$, and
\item $\Bcal_1= \{u_3\}$. \end{itemize} 
Since $x_3$ divides a minimal generator of $\ini(I)$, by Proposition  
\ref{CMcharacterization} we deduce that $R/I$ is not Cohen-Macaulay.
We compute $\delta_3 = {\rm min}\{\delta \, \vert \, u_3
x_{3}^{\delta} \in \ini(I)\}$ and get that $\delta_3 = 2$ and that
$r_3=-x_4u_4+(x_3^3x_4-x_3^2x_4^2)u_1$ is the remainder of the division
of $u_3 x_3^2$ by $\mathcal G$. Hence, following Theorem
\ref{algorithm_NoetherResolution}, we obtain the Noether resolution or
$R/I$:

$$\mathcal F: 0 \xrightarrow{} A(-8)  \xrightarrow{\psi}  \begin{array}{c} A
\oplus  A(-3)  \oplus  A(-4)  \oplus  \\ \oplus A(-6)  \oplus A(-7)
\oplus A(-9)\end{array} \xrightarrow{}   R/I \xrightarrow{}
0,$$ where $\psi$  is given by the matrix $$\left( \begin{array}{c}
-x_3^3x_4+x_3^2x_4^2 \\
 0 \\
x_3^2 \\
x_4 \\
 0 \\
 0 \\
 \end{array} \right)$$
Moreover, by Corollary \ref{HSm=2}, we obtain that the
weighted Hilbert series of $R/I$ is
 $$HS_{R/I}(t) = \frac{1+t^3+t^4+t^6+t^7-t^8+t^9}{(1-t^2)^2}.$$

If we consider the same parametric surface over an infinite field of
characteristic $2$, we obtain that $\{x_1^2+x_3^3+x_3^2x_4,
x_2^2+x_3x_4^3+x_4^4\}$ is a minimal Gr\"obner basis of $I$ with
respect to $>_{\omega}$, the weighted degree reverse lexicographic order
with $\omega = (3,4,2,2)$. Then we have that
$$\Bcal_0 = \{v_1:=1,v_2:=x_1,v_3:=x_2,v_4:=x_1x_2\},$$
and $\Bcal_1 = \emptyset$, so $R/I$ is Cohen-Macaulay.
Moreover, we also obtain the Noether resolution of $R/I$

$$\mathcal F': 0 \xrightarrow{}  A
\oplus  A(-3)  \oplus  A(-4)  \oplus A(-7)
\xrightarrow{}  R/I \xrightarrow{}
0$$
and the weighted Hilbert series of $R/I$ is
$$HS_{R/I}(t) = \frac{1+t^3+t^4+t^7}{(1-t^2)^2}.$$
\end{example}

\bigskip

To end this section, we consider the particular case where $I$ is 
standard graded homogeneous, i.e., $\omega = (1,\ldots,1)$. In this setting, we obtain a formula for the Castelnuovo-Mumford regularity of $R/I$ in terms of $\ini(I)$ or, more precisely, in terms of 
$\Bcal_0$ and $\Bcal_1$.  This formula is equivalent to that of \cite[Theorem 2.7]{BerGi99} provided $x_n$ is a nonzero divisor of $R/I$.

\begin{corollary}\label{regularity} Let $R/I$ be a $2$-dimensional standard graded ring such that $x_n$ is a nonzero divisor. Then, 
$$\reg(R/I)= \max\{\deg(v),\deg(u)+\delta_u-1\ \vert\ v\in\Bcal_0,u\in\Bcal_1\}$$
\end{corollary}

\medskip

In the following example we apply all the results of this section.

\begin{example}
Let $K$ be a characteristic zero field  and let us consider the projective curve $\Ccal$ of $\mathbb{P}_K^4$ parametrically defined by:
$$x_1=s^3t^5-st^7,x_2=s^7 t, x_3=s^4t^4,x_4=s^8,x_5=t^8.$$ A direct computation
with {\sc Singular}, {\sc CoCoA} or  {\sc Macaulay 2} yields that a  minimal Gr\"obner basis $\mathcal G$ of the defining ideal $I \subset R = K[x_1,\ldots,x_5]$ of $\Ccal$ with
respect to the degree reverse lexicographic order consists of $10$
elements and that
$$\ini(I)=(x_1^4,x_2^4,x_1^3x_3,x_1x_3x_4^2,x_1^2x_2,x_1x_2^2,x_1x_2x_3,x_2^2x_3,x_1^2x_4,x_3^2).$$

\noindent Then, we obtain that the set $\Bcal_0$ is the following

\begin{center}$\Bcal_0 = \{u_1:=1, u_2:=x_1,u_3:=x_2,u_4:=x_3,u_5:=x_1^2,u_6:=x_1x_2,u_7:=x_2^2,$ \\ $u_8:=x_1x_3,
u_9:=x_2x_3,u_{10}:=x_1^3,u_{11}:=x_2^3,u_{12}:=x_1^2x_3\}$\end{center}
and the ideal $J$ is
$$J=(x_1^2,x_1x_3,x_3^2,x_2^2x_3,x_2^4) \subset R.$$

Thus, $\Bcal_1 = \{u_5, u_8, u_{10}, u_{12}\}$. For $i \in \{
5,8,10,12\}$ we compute $\delta_i$, the minimum integer such that $u_i
x_4^{\delta_i}\in\ini(I)$ and get that $\delta_4 = \delta_{10} =
\delta_{12} = 1$ and $\delta_8 = 2$. If we set $r_i$ the remainder
of the division of $u_i x_4^{\delta_i}$ for all $i \in
\{4,8,10,12\}$, we get that

\begin{itemize}
\item $r_4 = - x_4x_5^2 b_1 + 2 x_4x_5 b_4 + x_5 b_6 + x_5 b_7$,
\item $r_8 =  x_4^2x_5 b_3+ x_5 b_{11}$,
\item $r_{10} = x_4^2x_5 b_2 + 3x_4x_5 b_8 + (x_5^2-x_4x_5)b_9$, and
\item $r_{12} = x_4^2x_5^2 b_1+ x_4x_5 b_6 + x_5^2 b_7.$
\end{itemize}

Hence, we obtain the following  minimal graded free resolution of
$R/I$


$$\mathcal F: 0 \xrightarrow{}  A(-3)  \oplus  A^3(-4)  \xrightarrow{\psi}   A \oplus  A^3(-1)  A^5(-2) \oplus A^3(-3)
\xrightarrow{}  R/I \xrightarrow{} 0,$$

 where $\psi$ is given by the matrix
 \begin{center}
 \[ \left( \begin{array}{cccc}
x_4x_5^2 & 0 & 0 & -x_4^2x_5^2 \\
0 & 0 & -x_4^2x_5 & 0 \\
0 & 0 & 0 & 0 \\
-2x_4x_5 & -x_4^2x_5 & 0 & 0 \\
x_4 & 0 & 0 & 0 \\
-x_5 & 0 & 0 & -x_4x_5 \\
-x_5 & 0 & 0 & -x_5^2 \\
0 & x_4^2 & -3x_4x_5 & 0 \\
0 & 0 &  x_4x_5-x_5^2 & 0 \\
0 & 0 & x_4 & 0 \\
0 & -x_5 & 0 & 0 \\
0 & 0 & 0 & x_4 \end{array} \right)\]
 \end{center}

Moreover, the Hilbert series of $R/I$ is 
$$HS_{R/I}(t) = \frac{1+3t+5t^2+2t^3-3t^4}{(1-t)^2}.$$
and ${\rm reg}(R/I)=\max\{3,4-1\}=3$.
\end{example}


\section{Noether resolution. Simplicial semigroup rings}

This section concerns the study of Noether resolutions in simplicial semigroup rings $R/I$, i.e., whenever $I =
I_{\Acal}$ with $\Acal = \{a_1,\ldots,a_n\} \subset \N^d$ and
$a_{n-d+i} = w_{n-d+i} e_i$ for all $i \in \{1,\ldots,d\}$, where
$\{e_1,\ldots,e_d\}$ is the canonical basis of $\N^d$.  In this setting, $R/I_{\Acal}$ is isomorphic to the semigroup ring $K[\Scal]$, where $\Scal$ is the simplicial semigroup generated by $\Acal$. When $K$ is infinite, $I_{\Acal}$ is the vanishing ideal of the variety given parametrically
by $x_i := t^{a_i}$ for all $i \in \{1,\ldots,n\}$ (see, e.g., \cite{Villarreal2}) and, hence, $K[\Scal]$ is the coordinate ring of a parametric variety.
 In this section we study the
multigraded Noether resolution of $K[\Scal]$ with respect to the
multigrading ${\rm deg}_{\Scal}(x_i) = a_i \in \Scal$; namely,
$$\Fcal: 0 \longrightarrow  \oplus_{s \in \Scal_p} A \cdot
s  \buildrel{\psi_p}\over{ \longrightarrow} \cdots
\buildrel{\psi_1}\over{ \longrightarrow}  \oplus_{s \in \Scal_0} A
\cdot s \buildrel{\psi_0}\over{ \longrightarrow} K[\Scal]
\longrightarrow 0.$$ where $\Scal_i \subset \Scal$ for all $i \in
\{0,\ldots,p\}$. We observe that this multigrading is a refinement
of the grading given by $\omega = (\omega_1,\ldots,\omega_n)$ with
$\omega_i := \sum_{j=1}^d a_{ij} \in \Z^+$; thus, $I_{\Acal}$ is
$\omega$-homogeneous and the results of the previous section also
apply here.

Our objective is to provide a description of this resolution in
terms of the semigroup $\Scal$. We completely achieve this goal when $K[\Scal]$ is Cohen-Macaulay (which includes the case $d = 1$)
and also when $d = 2$.

For any value of $d \geq 1$, the first step of the resolution corresponds
to a minimal set of generators of $K[\Scal]$ as $A$-module and is
given by the following well known result.

\begin{proposition}\label{S0} Let $K[\Scal]$ be a simplicial semigroup ring. Then, $$\Scal_0 = \left\{s\in \Scal \ \vert \ s- a_{i} \notin\Scal {\rm \
for \ all\ } i\in\{n-d+1,\ldots,n\}\right\}.$$ Moreover, $\psi_0: \oplus_{s
\in \Scal_0} A \cdot s \longrightarrow K[\Scal]$ is the homomorphism
of $A$-modules induced by $e_s \mapsto t^s$, where $\{e_s\, \vert
\, s \in \Scal_0\}$ is the canonical basis of $\oplus_{s \in
\Scal_0} A \cdot s$.
\end{proposition}

Proposition \ref{S0} gives us the whole multigraded Noether resolution of $K[\Scal]$ when
$K[\Scal]$ is Cohen-Macaulay.

\medskip
In  \cite[Theorem 1]{GotoSuzukiWatanabe76} (see also \cite[Theorem
6.4]{Stanley}), the authors provide a characterization of the
Cohen-Macaulay property of $K[\Scal]$. In the
following result we are proving an equivalent result that
characterizes this property in terms of
the size of $\Scal_0$. The proof shows how to
obtain certain elements of ${\rm Ker}(\psi_0)$ and this idea will be
later exploited to describe the whole resolution when $d = 2$ and $K[\Scal]$ is not 
Cohen-Macaulay.

\begin{proposition}\label{CMcharacSemigroup}Let $\Scal$ be a simplicial semigroup as above. Set $D := \left(\prod_{i = 1}^d \omega_{n-d+i}\right) /
[\Z^d:\Z\Scal]$, where $[\Z^d:\Z\Scal]$ denotes the index of the
group generated by $\Scal$ in $\Z^d$. Then, $K[\Scal]$ is
Cohen-Macaulay $\Longleftrightarrow$ $\vert \Scal_0 \vert= D$.
\end{proposition}
\begin{proof}
By Auslander-Buchsbaum formula we deduce that $K[\Scal]$ is
Cohen-Macaulay if and only if $\psi_0$ is injective, where $\psi_0$
is the morphism given in Proposition \ref{S0}. We are proving that
$\psi_0$ is injective if and only if $\vert \Scal_0 \vert = D$. We
define an equivalence relation on $\Z^d$, $u\sim v\
\Longleftrightarrow\
u-v\in\Z\{\omega_{n-d+1}e_1,\ldots,\omega_{n}e_d\}$. This relation
partitions $\Z \Scal$ into $D = [\Z\Scal:\Z\{\omega_{n-d+1}e_1,\ldots,\omega_{n}e_d\}]$ equivalence classes. Since 
$$\Z^d/\Z\,\Scal \simeq \left(\Z^d/\Z\{\omega_{n-d+1}e_1,\ldots,\omega_{n}e_d\}\right)/\left(\Z\,\Scal/\Z\{\omega_{n-d+1}e_1,\ldots,\omega_{n}e_d\}\right),$$ we get that  $D = \left(\prod_{i=1}^d \omega_{n-d+i}\right)/[\Z^d:\Z\Scal]$. Moreover, the
following two facts are easy to check: for every equivalence class
there exists an element $b \in \Scal_0$, and $\Scal = \Scal_0 + \N
\{\omega_{n-d+1}e_1,\ldots,\omega_{n}e_d\}$. This proves that $\vert
\Scal_0 \vert \geq D$.

Assume that  $\vert \Scal_0 \vert > D$, then there exist
$u,v\in\Scal_0$ such that $u\sim v$ or, equivalently,
$u+\sum_{i=1}^d\lambda_i\omega_{n-d+i}e_i=v +
\sum_{i=1}^d\mu_i\omega_{n-d+i}e_i$ for some $\lambda_i,\mu_i\in\N$
for all $i\in\{1,\ldots,d\}$. Thus $x_{n-d+1}^{\lambda_1}\cdots
x_n^{\lambda_d} e_u - x_{n-d+1}^{\mu_1}\cdots x_n^{\mu_d}e_v\in{\rm
Ker}(\psi_0)$ and $\psi_0$ is not injective. \hfill\break 
Assume now that $\vert
\Scal_0 \vert = D$, then for every $s_1,s_2 \in \Scal_0$, $s_1 \neq
s_2$, we have that $s_1 \not\sim s_2$. As a consequence, an element
$\rho \in \oplus_{s \in \Scal_0} A \cdot s$ is homogeneous if and
only if  it is a monomial, i.e., $\rho = c x^{\alpha} e_s$ for some
$c \in K$, $x^{\alpha} \in A$ and $s \in \Scal_0$. Since the image
by $\psi_0$ of a monomial is another monomial, then there are no
homogeneous elements in ${\rm Ker}(\psi_0)$ different from $0$, so
$\psi_0$ is injective.\end{proof}

\medskip

From now on suppose that $K[\Scal]$ is a $2$-dimensional non Cohen-Macaulay semigroup ring.
In this setting, we consider the set $$\Delta := \left\{s \in \Scal
\ \vert \ s - a_{n-1}, s - a_n \in\Scal\ {\rm and}\ s - a_n -
a_{n-1} \notin\Scal\right\}.$$  The set $\Delta$ or slight variants of it has been considered by other authors (see, e.g., \cite{GotoSuzukiWatanabe76,Stanley,TH}).
We claim that $\Delta$ has
exactly $\vert \Scal_0 \vert - D$ elements.  Indeed, if we consider the equivalence relation $\sim$ of Proposition \ref{CMcharacSemigroup}, then $\sim$ partitions $\Z \Scal$ in $D$ classes $C_1,\ldots,C_D$ and it is straightforward to check that $\vert \Delta \cap C_i\vert  = \vert \Scal_0 \cap C_i\vert  - 1$ for all $i \in \{1,\ldots,D\}$. From here, we easily deduce that $\vert \Delta\vert  = \vert \Scal_0\vert  - D$. 
Hence, a direct consequence of Proposition \ref{CMcharacSemigroup} is that $\Delta$ is nonempty because $K[\Scal]$ is not Cohen-Macaulay.
Furthermore, as Theorem
\ref{S_1} shows, the set $\Delta$ is not only useful to characterize
the Cohen-Macaulay property but also provides the set of shifts in
the second step of the multigraded Noether resolution of $K[\Scal]$.

\begin{theorem}\label{S_1}
 Let $K[\Scal]$ be a $2$-dimensional semigroup ring and let
$$\Delta = \left\{s \in \Scal \ \vert \ s - a_{n-1}, s - a_n
\in\Scal\ {\rm and}\ s - a_n - a_{n-1} \notin\Scal\right\},$$
as above. Then, $\Scal_1 = \Delta$.
\end{theorem}
\begin{proof}
Set $\Bcal_0$ the monomial basis of $R / ({\rm in}(I_{\Acal}), x_{n-1},x_n)$, 
where $\ini(I_{\Acal})$ is the initial ideal of $I_{\Acal}$ with respect to $>_{\omega}$.
For every $u = x_1^{\alpha_1} \cdots x_{n}^{\alpha_n} \in \Bcal_1$ we set $\delta_u \geq 1$ the minimum integer such that $u
x_{n-1}^{\delta_u} \in {\rm in}(I_{\Acal})$. Consider $p_u \in R$ the
remainder of $u x_{n-1}^{\delta_u}$ modulo the reduced Gr\"obner
basis of $I_{\Acal}$ with respect to $>_{\omega}$, then $u
x_{n-1}^{\delta_u} - p_u \in I_{\Acal}$. Since $I_{\Acal}$ is a binomial ideal, we
get that $p_u = x^{\gamma}$ for some $(\gamma_1,\ldots,\gamma_n) \in
\N^n$. Moreover, the condition $x^{\alpha}
> x^{\gamma}$ and the minimality of $\delta_u$ imply
that $\gamma_n > 0$ and $\gamma_{n-1} = 0$, so $x^{\gamma} = v_u
x_{n}^{\gamma_{v_u}}$ with $v_u \in \Bcal_0$. As we proved in Proposition
\ref{Cmonomials}, if we denote by $\{e_v \, \vert \, v \in \Bcal_0\}$
the canonical basis of $\oplus_{v \in \Bcal_0} A(-{\rm deg}_{\Scal}(v))$ and $h_u :=
x_{n-1}^{\delta_u} e_u - x_{n}^{\gamma_{v_u}} e_{v_u}$ for all $u
\in \Bcal_1$, then ${\rm Ker}(\psi_0)$ is the $A$-module
minimally generated by $\Ccal := \{h_u \, \vert \, u \in \Bcal_1\}.$ Let us prove that $$\{ {\rm deg}_{\Scal}(h_u)  \, \vert \, u
\in \Bcal_1\} = \{s \in \Scal \, \vert \, s - a_{n-1}, s - a_n
\in \Scal{\rm \ and \ } s - a_{n-1} - a_{n} \notin \Scal\}.$$ Take
$s = {\rm deg}_{\Scal}(h_u)$ for some $u \in \Bcal_1$, then $s
= {\rm deg}_{\Scal}(h_u) = {\rm deg}_{\Scal}(u) + \delta_u a_{n-1} =
{\rm deg}_{\Scal}(v_u) + \gamma_{v_u} a_n$. Since $\delta_u,
\gamma_{v_u} \geq 1$, we get that both $s - a_{n-1}, s - a_{n} \in
\Scal$. Moreover, if $s - a_{n-1} - a_n = \sum_{i = 1}^n \delta_i
a_i \in \Scal$, then $x_{n-1}^{\delta_u - 1} u - x^{\lambda} x_{n+1}
\in I_{\Acal}$, which contradicts the minimality of $\delta_u$.

Take now $s \in \Scal$ such that $s - a_{n-1}, s - a_n \in \Scal$
and $s - a_{n-1} - a_n \notin \Scal$. Since $s - a_{n-1},s -
a_{n}\in\Scal$, there exists $s',s'' \in\Scal_0$ and
$\gamma_1,\gamma_2, \lambda_1, \lambda_2 \in\N$ such that $s - a_n =
s' + \gamma_1 a_{n-1} + \gamma_2 a_{n}$ and  $s - a_{n+1} = s'' +
\lambda_1 a_{n-1} + \lambda_2 a_{n}$. Observe that $\gamma_2=0$,
otherwise $s - a_{n-1} - a_n = s' + \gamma_1
a_{n-1}+(\gamma_2-1)a_{n}\in\Scal$, a contradiction. Analogously
$\lambda_1 = 0$. Take $u, v  \in \Bcal_0$ such that ${\rm
deg}_{\Scal}(u) = s'$ and ${\rm deg}_{\Scal}(v) = s''$. We claim
that $u \in J$ and that $\delta_u = \gamma_1$. Indeed, $f:= u
x_{n-1}^{\gamma_1} - v x_n^{\lambda_2} \in I_{\Acal}$ and ${\rm in}(f) = u
x_{n-1}^{\gamma_1}$, so $u \in \Bcal_1$. Moreover, if there
exists $\gamma' < \delta_u$, then $s - a_{n-1} - a_n \in \Scal$, a
contradiction.
\end{proof}

One of the interests of Proposition \ref{CMcharacSemigroup} and
Theorem \ref{S_1} is that they describe multigraded Noether resolutions of
dimension $2$ semigroup rings in terms of the semigroup $\Scal$ and,
 in particular, they do not depend on the characteristic of the field $K$. 

\medskip
Now we consider the multigraded Hilbert Series of $K[\Scal]$, which
is defined by
$$HS_{K[\Scal]}(t) = \sum_{s \in\Scal}t^s =\sum_{s = (s_1,\ldots,s_d) \in\Scal} t_1^{s_1}\cdots t_d^{s_d},$$
When $d = 2$, from the description of the multigraded Noether resolution of
$K[\Scal]$ we derive an expression of its multigraded Hilbert series
in terms of $\Scal_0$ and $\Scal_1$.
\begin{corollary}\label{multiHilbert}
Let $K[\Scal]$ be a dimension $2$ semigroup ring. The multigraded Hilbert series of $K[\Scal]$ is:
$$HS_{K[\Scal]}(t)= \frac{\sum_{s \in\Scal_0}t^{s} - \sum_{s\in\Scal_1}t^{s}}{(1-t_{1}^{\omega_{n-1}})(1-t_2^{\omega_n})}.$$
\end{corollary}
\medskip

\begin{remark}\label{multtostandardgraded}
When $K[\Scal]$ is a two dimensional semigroup ring and $\Scal$ is generated by the
set $\Acal = \{a_1,\ldots,a_n\} \subset \N^2$, if we set $\omega = (\omega_1,\ldots,\omega_n) \in \N^n$ with
$\omega_i := a_{i,1} + a_{i,2}$ for all $i \in \{1,\ldots,n\}$,
then $I_{\Acal}$ is $\omega$-homogeneous, as observed at the beginning of this section. The Noether resolution of $K[\Scal]$ with respect to this grading is easily obtained 
from the multigraded one. Indeed, it is given by
the following expression:
$$\mathcal F: 0\longrightarrow  \oplus_{(b_1,b_2) \in \Scal_1} A(-(b_1+b_2))
\buildrel{\psi_1}\over{ \longrightarrow}   \oplus_{(b_1,b_2) \in \Scal_0} A(-(b_1+b_2)) \buildrel{\psi_0}\over{\longrightarrow}
K[\Scal] \longrightarrow 0.$$
In addition, the weighted Hilbert series of $K[\Scal]$ is obtained from
the multigraded one by just considering the transformation
$t_1^{\alpha_1} t_2^{\alpha_2} \mapsto t^{\alpha_1 + 
\alpha_2}$.

 When $\omega_1 = \cdots = \omega_n$, then $I_{\Acal}$ is a
homogeneous ideal. In this setting, the Noether resolution with respect to the standard grading is 
$$\mathcal F: 0\longrightarrow  \oplus_{(b_1,b_2) \in \Scal_1} A(-(b_1+b_2)/\omega_1)
\buildrel{\psi_1}\over{ \longrightarrow}   \oplus_{(b_1,b_2) \in \Scal_0} A(-(b_1+b_2)/\omega_1) \buildrel{\psi_0}\over{\longrightarrow}
K[\Scal] \longrightarrow 0.$$ 
Thus, the Castelnuovo-Mumford regularity of $K[\Scal]$ is \begin{equation}\label{formularegularidad} \reg(K[\Scal])=\max\left( \left\{\frac{b_1+b_2}{\omega_1}\ \vert \ (b_1,b_2)
\in\Scal_0\right\} \cup \left\{\frac{b_1+b_2}{\omega_1} - 1\ \vert\
(b_1,b_2) \in\Scal_1\right\}\right). \end{equation}
Moreover, the Hilbert series of $K[\Scal]$ is obtained from
the multigraded Hilbert series by just considering the transformation
$t_1^{\alpha_1} t_2^{\alpha_2} \mapsto t^{(\alpha_1 + 
\alpha_2)/\omega_1}$. 
\end{remark}

\section{Macaulayfication of simplicial semigroup rings}
 Given $K[\Scal]$ a simplicial semigroup ring, the semigroup ring $K[\Scal']$ is the {\it Macaulayfication} of $K[\Scal]$ if the 
three following conditions are satisfied: \begin{enumerate} \item $\Scal \subset \Scal'$, \item $K[\Scal']$ is Cohen-Macaulay, and
\item the Krull dimension of $K[\Scal' \setminus \Scal]$ is $\leq d-2$, where $d$ is the Krull dimension of $K[\Scal]$. 
\end{enumerate}
The existence and uniqueness of a $K[\Scal']$ fulfilling the previous properties for simplicial semigroup rings is guaranteed by \cite[Theorem 5]{Morales07}.
In this section we describe explicitly the Macaulayfication of any
simplicial semigroup ring in terms of the set $\Scal_0$. For this
purpose we consider the same equivalence relation in $\Z^d$ as in proof of Proposition \ref{CMcharacSemigroup}, namely, for $s_1,s_2 \in \Z^d$
$$s_1\sim s_2 \Longleftrightarrow s_1-s_2\in\Z\{\omega_{n-d+1}e_1,\ldots,\omega_ne_d\}.$$
 As we have seen,
 $\Scal_0 \subset \Z^d$ is partitioned into $D:= \omega_{n-d+1}\cdots
\omega_n/[\Z^d:\Z\Scal]$ equivalence classes $S^1,\ldots,S^D$. For
every equivalence class $S^i$ we define a vector $b_i$ in the
following way. We take $S^i = \{s_1,\ldots,s_t\}$, where $s_j =
(s_{j1},\ldots,s_{jd}) \in \N^d$ for all $j \in \{1,\ldots,t\}$ and
define $b_i = (b_{i1},\ldots,b_{id}) \in \N^d$ as the vector whose
$k$-th coordinate $b_{ik}$ equals the minimum of the $k$-th
coordinates of $s_1,\ldots,s_t$, this is, $b_{ik} := {\rm
min}\{s_{1k},\ldots,s_{tk}\}$. We denote $\frak B := \{b_1,\ldots,b_D\}$ and
\begin{equation} \label{semimacaulification} \Scal' := \frak B + \N\{\omega_{n-d+1}e_1,\ldots,\omega_ne_d\}. \end{equation}

The objective of this section is to prove that $K[\Scal']$ is the
Macaulayfication of $K[\Scal]$. The main issue in the proof is to show that ${\rm
dim}(K[\Scal' \setminus \Scal]) \leq d-2$.  For this purpose we use a technique developed in \cite{M-N}
which consists of determining the dimension of a graded ring by studying its Hilbert function.
More precisely, for $L$ an $\omega$-homogeneous ideal, if we denote by $h(i)$ the Hilbert
function of $R/L$, by \cite[Lemma 1.4]{MoralesCRAS}, there exist some polynomials $Q_{1},..., Q_{s} \in \Z[t]$ with $s \in \Z^+$ such that  
 $h(ls+i)=Q_i(l)$   for all $i \in \{1,\ldots,s\}$ and $l \in \Z^+$ large enough.
Moreover, in \cite{MoralesPreprint}, the author proves the following.

\begin{theorem}\label{dimensionhilbert}Let $L$ be a $\omega$-homogeneous ideal and denote by $h: \N \rightarrow \N$ the Hilbert function of $R/L$.
If we set $h^0(n) = \sum_{i = 0}^n h(i)$, then there exist $s$ polynomials
$f_1,...,f_s \in \Z[t]$  such that $h^0(ls+i)=f_i(l)$ for all $i \in \{1,\ldots,s\}$ and $l \in \Z^+$ large enough.  Moreover, all these polynomials $f_1,\ldots,f_s$ have the same leading term $c\, t^{{\rm dim}(R/L)}/ ({\rm dim}(R/L))!$ with $c \in \Z^+$. 
\end{theorem}

In the proof of Theorem \ref{macaulayfication_semigroups}, we relate the Hilbert function of $K[\Scal' \setminus \Scal]$ with that of 
several monomial ideals and use of the following technical lemma.

\begin{lemma}\label{keylemmaMacaulayfication}
Let $M \subset K[y_1,\ldots,y_d]$ be a monomial ideal. If for all
$i\in\{1,\ldots,d\}$ there exist $x^{\alpha}\in M$ such that
$x_i\nmid x^{\alpha}$, then ${\rm dim}(K[y_1,\ldots,y_d]/M) \leq d-2$.
\end{lemma}
\begin{proof}
Let us prove that $M$ has height $\geq 2$. By contradiction, assume
that $M$ has an associated prime $\mathfrak P$ of height one. Since
$M$ is monomial, then so is $\mathfrak P$. Therefore, $\mathfrak P =
(x_i)$ for some $i\in\{1\ldots,d\}$. Hence we get that $M \subset
\sqrt{M} \subset \mathfrak P = (x_i)$, a contradiction.
\end{proof}

Now we can proceed with the proof of the main result of this section.

\begin{theorem}\label{macaulayfication_semigroups} Let $K[\Scal]$ be a simplicial semigroup ring and let $\Scal'$ be the semigroup described in (\ref{semimacaulification}). Then,
$K[\Scal']$ is the Macaulayfication of $K[\Scal]$.
\end{theorem}
\begin{proof}
Is is clear that $\Scal \subset \Scal'$. In order to obtain the result it suffices to prove that $\Scal'$ is a semigroup, that $K[\Scal'] $ is Cohen-Macaulay and that ${\rm dim}(K[\Scal' \setminus \Scal]) \leq {\rm dim}(K[\Scal]) -
2$ (see, e.g., \cite{Morales07}).
 
Let us first prove that $\Scal'$ is a semigroup. Take $s_1, s_2 \in
\Scal'$, then there exists $i, j \in \{1,\ldots,D\}$ such that $s_1=
b_i + c_1$ and $s_2 = b_j+ c_2$ for some $c_1,c_2 \in \N
\{\omega_{n-d+1} e_1,\ldots,\omega_n e_d\}$. Then $s_1 + s_2 = b_i +
b_j + c_1 + c_2$. We take $k \in \{1,\ldots,D\}$ such that $b_k \sim
b_i + b_j$. By construction of $\frak B$ we have that $b_{k} = b_i + b_j + c_3$ for
some $c_3 \in \{\omega_{n-d+1} e_1,\ldots,\omega_n e_d\}$ and, hence, $s_1 +
s_2 \in \Scal'$.

To prove that $\Scal'$ is Cohen-Macaulay it suffices to observe that
$\frak B = \{b \in \Scal' \, \vert \, b - a_i\notin \Scal'$ for all
$i \in \{1,\ldots,d\}\}$ and that $\vert \frak B \vert = D$, so by
Proposition \ref{CMcharacSemigroup} it follows that $\Scal'$ is
Cohen-Macaulay.

Let us prove that $\dim(K[\Scal'\setminus\Scal]) \leq d-2$.  For all $s = (s_1,\ldots,s_m) \in \N^m$
we consider the grading ${\rm deg}(t^s) = \sum_{i = 1}^m s_i$ and we denote $h, \, h'$
and  $\widehat{h}$ the Hilbert functions of $K[\Scal],\, K[\Scal']$
and  $K[\Scal' \setminus \Scal]$ respectively, then $\widehat{h} =
h' - h$. Moreover, we have that $h' = \sum_{i=1}^{D}h_i'$ and $h =
\sum_{i=1}^{D}h_i$ where $h_i'(d) :=\left\vert \{s\in\Scal'\ \vert\
\deg t^s = d\ {\rm and}\ s\sim b_i\}\right\vert$ and $h_i(d)
:=\left\vert \{s\in\Scal\ \vert\ \deg t^s = d\ {\rm and}\ s\sim
b_i\}\right\vert$. For each $i\in\{1,\ldots,D\}$ we define a
monomial ideal $M_i \subset k[y_1,\ldots,y_d]$ as follows: for every
$b \in \Scal$ such that $b \sim b_i$ we define the monomial $m_b :=
y_1^{\beta_1} \cdots y_d^{\beta_d}$ if $b = b_i + \sum_{i = 1}^d
\beta_i \omega_{n-d+i} e_i$ and $M_i := (\{m_b \, \vert \, b \in
\Scal, b \sim b_i\})$. We consider in $K[y_1,\ldots,y_d]$ the
grading ${\rm deg}_{\omega}(y_i) = \omega_{n-d+i}$ and denote by
$h^{\omega}_i$ the corres\-pon\-ding $\omega$-homogeneous Hilbert
function of $K[y_1,\ldots,y_d]/M_i$. We have the following equality
$h^{\omega}_i(\lambda) = h_i'(\sum_{j = 1}^d b_{ij} + \lambda) -
h_i(\sum_{j = 1}^d b_{i_j} + \lambda)$ because $y^{\beta}\notin M_i
\Longleftrightarrow b_i + \sum_{i=1}^d \beta_i \omega_{n-d+i} e_i
\in \Scal' \setminus \Scal$. Hence, we have expressed the Hilbert
function $\widehat{h}$ of $K[\Scal \setminus \Scal']$ as a sum of
$D$ Hilbert functions of $K[y_1,\ldots,y_d] / M_i$, for some
monomial ideals $M_1,\ldots,M_D$ and, by Lemma
\ref{keylemmaMacaulayfication}, ${\rm dim}(K[y_1,\ldots,y_d] / M_i)
\leq d-2$. Thus, by Theorem \ref{dimensionhilbert}, we can conclude that the dimension of 
$K[\Scal' \setminus \Scal]$ equals the maximum of ${\rm
dim}(K[y_1,\ldots,y_d] /M_i) \leq d-2$ and we get the
result.\end{proof}

We finish this section with an example showing how to compute the
Macaulayfication by means of the set $\Scal_0$. Moreover, this
example illustrates that even if $K[\Scal] = R/I_{\Acal}$ with $I_{\Acal}$ a
homogeneous ideal, it might happen that the ideal associated to
$K[\Scal']$ is not standard homogeneous.

\begin{example}We consider the semigroup ring $K[\Scal]$, where $\Scal \subset \N^2$ is the semigroup
 generated by $\Acal := \{ (1,9),(4,6),(5,5),(10,0),(0,10)
\} \subset\N^2$. Then, $K[\Scal] = R/I_{\Acal}$ and $I_{\Acal}$ is
homogeneous. If we compute the set $\Scal_0$ we get that $$\Scal_0 =
\left\{(0,0),(1,9),(2,18),(3,27),(13,17),(4,6),(5,5),(6,14),(7,23),(8,12),(9,11)\right\}.$$
Moreover we compute $D = 100 / [\Z^d : \Z \Scal] = 10$ and get
$S^{1}=\left\{(0,0)\right\}$, $S^{2}=\left\{(1,9)\right\}$,
$S^{3}=\left\{(2,18)\right\}$,
$S^{4}=\left\{(3,27),(13,17)\right\}$, $S^{5}=\left\{(4,6)
\right\}$, $S^{6}=\left\{(5,5)\right\}$,
$S^{7}=\left\{(6,14)\right\}$,
$S^{8}=\left\{(7,23)\right\}$, $S^{9}=\left\{(8,12)\right\}$
and  $S^{10}=\left\{(9,11)\right\}$. So, the Macaulayfication
$K[\Scal']$ of $K[\Scal]$ is given by $\Scal' = \frak B + \N
\{(10,0),(0,10)\}$, where $$\frak B = \{
(0,0),(1,9),(2,18),(3,17),(4,6),(5,5),(6,14),(7,23),(8,12),(9,11)\}.$$
Or equivalently, $\Scal'$ is the semigroup generated by
$$\Acal' =
\{(1,9), (3,17), (4,6), (5,5), (10,0), (0,10)\}.$$ We observe that
$K[\Scal'] \simeq K[x_1,\ldots,x_6] / I_{\Acal'}$ and that $I_{\Acal'}$
is $\omega$-homogeneous with respect to $\omega = (1,2,1,1,1,1)$ but not
standard homogeneous.
\end{example}

\section{An upper bound for the Castelnuovo-Mumford regularity of projective monomial curves}

Every sequence $m_1 < \ldots < m_n$ of relatively prime positive integers with $n \geq 2$ has
associated the projective monomial curve $\Ccal \subset {\mathbb
P}_K^n$ given parametrically by $x_i := s^{m_i}  t^{m_n-m_i}$ for
all $i \in \{1,\ldots,n-1\},\, x_{n} = s^{m_n},  \, x_{n+1} :=
t^{m_n}$. If we set $\Acal := \{a_1,\ldots,a_{n+1}\} \subset \N^2$
where $a_i := (m_i, m_n - m_i), a_n := (m_n, 0)$ and $a_{n+1} :=
(0,m_n)$, it turns out that $I_{\Acal} \subset K[x_1,\ldots,x_{n+1}]$ is the defining ideal of
$\Ccal$. If we denote by $\Scal$ the semigroup generated by
$\Acal$, then the $2$-dimensional
semigroup ring  $K[\Scal]$ is isomorphic to $K[x_1,\ldots,x_{n+1}]/I_{\Acal}$, the homogeneous coordinate ring of $\Ccal$.
Hence, the methods of the previous sections apply here
to describe its multigraded Noether resolution, and the formula (\ref{formularegularidad}) in Remark \ref{multtostandardgraded} for the Castelnuovo-Mumford regularity  
holds in this context (with $\omega_1 = m_n$). The goal of this section is to use this formula to prove Theorem \ref{upperboundregmoncurve}, which provides an
upper bound for the Castelnuovo-Mumford regularity of $K[\Scal]$.
The proof we are presenting is elementary and uses some classical results on numerical semigroups. We will introduce now the results on numerical semigroups that we need for
our proof (for more on this topic we refer to \cite{Rosaleslibro} and \cite{RA}). 

Given $m_1,\ldots,m_n$ a set
of relatively prime integers, we denote by $\mathcal R$ the numerical subsemigroup of $\N$ spanned by $m_1,\ldots,m_n$. The largest integer that does not belong to $\mathcal R$
is called the {\it Frobenius number} of $\mathcal R$ and will be denoted by ${\rm g}(\mathcal R)$. We consider the {\it Apery set of $\mathcal R$ with respect
to $m_n$}, i.e., the set $${\rm Ap}(\mathcal R,m_n) := \{a \in \mathcal R \, \vert \, a - m_n \notin \mathcal R\}.$$ It is a well known and easy to check that
 ${\rm Ap}(\mathcal R,m_n)$ constitutes a full set of residues modulo $m_n$ (and, in particular, has $m_n$ elements) and that ${\rm max}({\rm Ap}(\mathcal R,m_n)) = {\rm g}(\mathcal R) + m_n.$
 We will also use an upper bound on ${\rm g}(\mathcal R)$ which is a slight variant of the one given in \cite{Selmer} (which was deduced from 
 a result of \cite{EG}). The reason why we do not use Selmer's bound itself is that it is only valid under the additional hypothesis that $n \leq m_1$. This is not a restrictive hypothesis when studying numerical semigroups, because whenever $m_1 < \cdots < m_n$ is a minimal set of generators of $\mathcal R$, then $n \leq m_1$. In our current setting of projective monomial
 curves, the case where $m_1 < \cdots < m_n$ is not a minimal set of generators of $\mathcal R$ is interesting by itself (even the case $m_1 = 1$ is interesting); hence, a direct
 adaptation of the proof of Selmer yields that \begin{equation}\label{upperboundfro}g(\mathcal R) \leq 2 m_n \left\lfloor \frac{m_{\tau}}{n} \right\rfloor - m_{\tau},\end{equation} for every $m_{\tau} \geq n$. Note that $m_n \geq n$ and then, such a value $\tau$ always exists.. 
 
 We first include a result providing an upper bound for  ${\rm reg}(K[\Scal])$ when $K[\Scal]$ is Cohen-Macaulay.
 
 \begin{proposition}\label{cotareguCM} Let $m_1 < \ldots < m_n$ be a sequence of relatively prime positive integers with $n \geq 2$ and let $\tau \in \{1,\ldots,n\}$ such that
 $m_{\tau} \geq n$. If $K[\Scal]$ is Cohen-Macaulay, then $${\rm reg}(K[\Scal]) \leq \left\lfloor  (2m_n \left\lfloor \frac{m_{\tau}}{n} \right\rfloor - m_{\tau} + m_n)/m_1\right\rfloor .$$ In particular, if 
 $m_1 \geq n$, we have that ${\rm reg}(K[\Scal]) \leq \left\lfloor m_n\left(\frac{2}{n} + \frac{1}{m_1}\right) - 1 \right\rfloor.$
 \end{proposition}
 \begin{proof}
We consider the equivalence relation $\sim$ of Section $4$. Indeed, since now $\Z \Scal = \{(x,y) \, \vert \, x + y \equiv 0\ ({\rm mod}\ m_n)\}$, then we have that $\sim$ partitions
the set $\Scal_0$ in exactly $m_n$ equivalence classes. Moreover, since $K[\Scal]$ is Cohen-Macaulay,  we have that 
\begin{itemize}
\item each of these classes has a unique element,
\item $\Scal_1 = \emptyset$, and
\item  ${\rm reg}(K[\Scal]) = {\rm max}\left\{\frac{b_1+b_2}{m_n} \, \vert \, (b_1,b_2) \in \Scal_0\right\}$ (see Remark 
\ref{multtostandardgraded}).
\end{itemize}
Let us take $(b_1,b_2) \in \Scal_0$, then  $(b_1,b_2) = \sum_{i = 1}^{n-1} \alpha_i a_i$ and $(b_1 + b_2)/m_n = \sum_{i = 1}^{n-1} \alpha_i$. Moreover,  we claim that 
$b_1 \in {\rm Ap}(\mathcal R,m_n)$. Otherwise, $b_1 - m_n \in \mathcal R$ and there would be another element $(c_1,c_2) \in \Scal_0$ such that $(c_1,c_2) \sim (b_1,b_2)$, a contradiction.
Hence,  by (\ref{upperboundfro}),
$$\left(\sum_{i = 1}^{n-1} \alpha_i\right) m_1 \leq \sum_{i = 1}^{n-1} \alpha_i m_i = b_1 \leq {\rm g}(\mathcal R) + m_n \leq 2 m_n \left\lfloor \frac{m_{\tau}}{n} \right\rfloor - m_{\tau} + m_n.$$  And,
from this expression we conclude that $$\frac{b_1 + b_2}{m_n} =  \sum_{i = 1}^{n-1} \alpha_i \leq (2m_n \left\lfloor \frac{m_{\tau}}{n} \right\rfloor - m_{\tau} + m_n)/m_1.$$
When $m_1 \geq n$, then it suffices to take $\tau = 1$ to get the result. \end{proof}

 And now, we can prove the main result of the section. 
 
  \begin{theorem}\label{upperboundregmoncurve} Let $m_1 < \ldots < m_n$ be a sequence of relatively prime positive integers with $n \geq 2$. 
If we take $\tau, \lambda$ such that $m_{\tau} \geq n$ and $m_n - m_{\lambda} \geq n$. 
  Then, $${\rm reg}(K[\Scal]) \leq \left\lfloor  \frac{\left(2m_n \left\lfloor \frac{m_{\tau}}{n} \right\rfloor - m_{\tau} + m_n\right)}{m_1} + \frac{\left(2m_n \left\lfloor \frac{m_n - m_{\lambda}}{n} \right\rfloor + m_{\lambda}\right)}{(m_n - m_{n-1})} \right\rfloor  - 2.$$ In particular, if $m_1 \geq n$ and $m_n - m_{n-1} \geq n$, then ${\rm reg}(K[\Scal]) \leq  \left\lfloor m_n \left(\frac{4}{n} +  \frac{1}{m_1} + \frac{1}{m_n - m_{n-1}}\right) \right\rfloor -4$.
 \end{theorem}
 \begin{proof}
We consider $E$ one of the equivalence classes of $\Z \Scal$ induced by the equivalence relation $\sim$. First, assume that $\Scal_0 \cap E$ has a unique element which we call $(b_1,b_2)$. Then, $\Scal_1 \cap E = \emptyset$, and the same argument as in the proof of Proposition \ref{cotareguCM} proves that  $\frac{b_1+b_2}{m_n} \leq (2m_n \left\lfloor \frac{m_{\tau}}{n} \right\rfloor - m_{\tau} + m_n)/m_1.$  
 
Assume now that $\Scal_0 \cap E = \{(x_1,y_1), \ldots, (x_r,y_r)\}$ with $r \geq 2$ and $x_1 < x_2 < \cdots < x_r$. We claim that the following properties hold: \begin{itemize} 
\item[(a)] $x_1 \equiv x_2 \equiv \cdots \equiv x_r \ ({\rm mod}\ m_n),$
\item[(b)] $y_1 > \cdots > y_r$ and $y_1 \equiv y_2 \equiv \cdots \equiv y_r \ ({\rm mod}\ m_n),$
\item[(c)] $x_1 \in {\rm Ap}(\mathcal R,m_n)$, 
\item[(d)] $y_r \in {\rm Ap}(\mathcal R',m_n)$, where $\mathcal R'$ is the numerical semigroup generated by $m_n - m_{n-1} < m_n - m_{n-2} < \cdots < m_n - m_1 < m_n$, 
\item[(e)] $\Scal_1 \cap E  = \{(x_2,y_1), (x_3,y_2), \ldots, (x_r,y_{r-1})\},$ and
\item[(f)] ${\rm max}\left\{\frac{b_1 + b_2}{m_n} \, \vert \, (b_1,b_2) \in \Scal_0 \cap E\right\} \leq {\rm max}\left\{\frac{b_1 + b_2}{m_n} \, \vert \, (b_1,b_2) \in \Scal_1 \cap E\right\} - 1.$
\end{itemize}
Properties (a) and (b) are evident. To prove (c) and (d) it suffices to take into account the following facts: $\Scal \subset \mathcal R \times \mathcal R'$, and for every $b_1 \in \mathcal R$, $b_2 \in \mathcal R'$ there exist $c_1,c_2 \in \N$ such that $(b_1,c_2), (c_1,b_2) \in \Scal$. To prove (e) we first observe that 
\[ \Scal \cap E = \{b + \lambda (m_n,0) + \mu (0,m_n) \, \vert b \in \Scal_0 \cap E,\, \lambda, \mu \in \N\}.\]
Take now $(x,y) \in \Scal_1 \cap E$ and we take the minimum value $i \in \{1,\ldots,r\}$ such that $(x,y) = (x_i,y_i) + \lambda (m_n,0) + \mu (0,m_n)$ with $\lambda, \mu \in \N$; we observe that
\begin{itemize}
\item $\lambda > 0$; otherwise $(x,y) - (m_n,0) \notin \Scal$,
\item $\mu = 0$; otherwise $(x,y) - (m_n,m_n) = (x_i,y_i) + (\lambda-1)(m_n,0) + (\mu-1) (0,m_n) \in \Scal$, a contradiction,
\item $y \geq y_{r-1}$; otherwise $i = r$ and, since $(x,y) - (0,m_n) \in \Scal \cap E$, we get that $\mu \geq 1$, 
\item $x \leq x_{i+1}$; otherwise $(x,y) = (x_{i+1},y_{i+1}) + \lambda' (m_n,0) + \mu' (0,m_n)$ with $\lambda',\mu' \geq 1$, a contradiction, and
\item $x \geq x_{i+1}$; otherwise $(x,y) - (0,m_n) \notin \Scal$. 
\end{itemize}
Hence, $(x,y) = (x_{i+1},y_i)$ and $\Scal_1 \cap E \subseteq \{(x_2,y_1), (x_3,y_2), \ldots, (x_r,y_{r-1})\}$. Take now $i \in \{1,\ldots,r-1\}$, and consider $(x_{i+1},y_i) \in \Scal$. Since $(x_i,y_i), (x_{i+1},y_{i+1}) \in E$, $x_i \equiv x_{i+1}\ ({\rm mod}\ m_n)$ and $y_i \equiv y_{i+1}\ ({\rm mod}\ m_n)$, then  $(x_{i+1},y_i) \in E$. We also have that  there exist $\gamma, \delta \in \N$ such that $(x_{i+1},y_i) - (m_n,0) = (x_i,y_i) + \gamma (m_n,0) \in \Scal$ and $(x_{i+1},y_i) - (0,m_n) = (x_{i+1},y_{i+1}) + \delta (0,m_n) \in \Scal$. We claim that $(x_{i+1},y_i) - (m_n,m_n) \notin \Scal$. Otherwise there exists $j \in \{1,\ldots,r\}$ such that  $(x_{i+1}- m_n,y_i - m_n) = (x_j,y_j) + \lambda'(m_n,0) + \mu'(0,m_n)$; this is not possible since $x_{i+1} - m_n < x_{i+1}$ implies that $j \leq i$, and $y_i - m_n < y_i$ implies that $j \geq i+1$. Thus, $(x_{i+1},y_i) \in \Scal_1$ and (e) is proved.
Property (f) follows from (e).

Moreover, since $x_1 \in {\rm Ap}(\mathcal R,m_n)$, the same argument as in Proposition \ref{cotareguCM} proves that
\begin{equation}\label{ine1} \frac{x_1 + y_1}{m_n} \leq \left(2m_n \left\lfloor \frac{m_{\tau}}{n} \right\rfloor - m_{\tau} + m_n\right)/m_1,\end{equation}
and a similar argument with $y_r \in {\rm Ap}(\mathcal R',m_n)$ proves that 
\begin{equation}\label{ine2} \frac{x_r + y_r}{m_n} \leq \left(2m_n \left\lfloor \frac{m_n - m_{\lambda}}{n} \right\rfloor + m_{\lambda}\right)/(m_n - m_{n-1}).\end{equation}
And, since, 
\begin{equation}\label{ine3} \frac{x_{i+1} + y_i}{m_n} - 1  \leq \frac{x_r + y_1}{m_n} - 1 \leq 
\frac{x_1 + y_1}{m_n} + \frac{x_r + y_r}{m_n} - 2, \end{equation}
putting together (\ref{ine1}), (\ref{ine2}) and (\ref{ine3}) we get the result. If $m_1 \geq n$ and $m_n - m_{n-1} \geq n$, it suffices to take $\tau = 1$
and $\lambda = n-1$ to prove the result.
\end{proof}

It is not difficult to build examples such that the bound provided by Theorem \ref{upperboundregmoncurve} outperforms the bound of L'vovsky's. Let us see an example. 

\begin{example} 
Set $n \geq 6$ and consider $m_i = n + i$ for all $i \in \{1,\ldots,n-1\}$ and $m_n = 3n$, then we can take $\tau = 1$ and $\lambda = n-1$ and apply Theorem \ref{upperboundregmoncurve} to prove that $${\rm reg}(K[\Scal]) \leq  \left\lfloor 3n \left(\frac{4}{n} +  \frac{1}{n+1} + \frac{1}{n+1}\right) \right\rfloor -4 = 13,$$  meanwhile the result of L'vovsky provides an upper bound of $2n+1$.
\end{example}

\section{Noether resolution and Macaulayfication of projective monomial curves associated to arithmetic sequences and their canonical projections.}

Every sequence $m_1 < \ldots < m_n$ of positive integers with $n \geq 2$ has
associated the projective monomial curve $\Ccal \subset {\mathbb
P}_K^n$ given parametrically by $x_i := s^{m_i}  t^{m_n-m_i}$ for
all $i \in \{1,\ldots,n-1\},\, x_{n} = s^{m_n},  \, x_{n+1} :=
t^{m_n}$. If we set $\Acal := \{a_1,\ldots,a_{n+1}\} \subset \N^2$
where $a_i := (m_i, m_n - m_i), a_n := (m_n, 0)$ and $a_{n+1} :=
(0,m_n)$, it turns out that $I_{\Acal} \subset K[x_1,\ldots,x_{n+1}]$ is the defining ideal of
$\Ccal$. Moreover, if we denote by $\Scal$ the semigroup generated by
$\Acal$, then $K[\Scal] \simeq K[x_1,\ldots,x_{n+1}]/I_{\Acal}$ is a dimension $2$
semigroup ring and the methods of the previous sections apply here
to describe its multigraded Noether resolution.

In \cite{LiPatilRoberts12}, the authors studied the set $\Scal_0$
whenever $m_1<\cdots< m_n$ is an arithmetic sequence of relatively
prime integers, i.e., there exist $d, m_1 \in \Z^+$ such that $m_i =
m_1 + (i - 1)\,d$ for all $i\in\{1,\ldots,n\}$ and
$\gcd\{m_1,d\}=1$. In particular, they obtained the following
result.

\begin{theorem}{\cite[Theorem 3.4]{LiPatilRoberts12}}\label{lipatil}

$\Scal_0 = \left\{\left(\left\lceil\frac{j}{n-1}\right\rceil
m_n-jd,\,jd \right) \ \vert\ j\in\{0,\ldots,m_n-1\}\right\}$
\end{theorem}

From the previous result and Proposition \ref{CMcharacSemigroup} we deduce that $K[\Scal]$ is Cohen-Macaulay (see also \cite[Corollary 2.3]{BerGarGar15}), we obtain the shifts of the only step of the multigraded Noether resolution
and, by Corollary \ref{regularity}, we also derive that ${\rm
reg}(K[\Scal]) = \lceil (m_n - 1)/(n-1) \rceil$ (see also
\cite[Theorem 2.7]{BerGarGar15}). In the rest of this section we are using the
tools developed in the previous sections to study the canonical
projections of $\Ccal$, i.e., for all $r \in \{1,\ldots,n-1\}$ and $n \geq 3$ we
aim at studying the curve $\Ccal_r := \pi_r(\Ccal)$ obtained as the image of $\Ccal$ under the projection $\pi_r$ from $ \mathbb{P}_K^n $ to 
$\mathbb{P}_K^{n-1}$ defined by $(p_1:\cdots:p_{n+1}) \mapsto
(p_1:\cdots:p_{r-1}:p_{r+1}:\cdots:p_{n+1})$. We know that the
vanishing ideal of $\Ccal_r$ is $I_{\Acal_r}$, where $\Acal_r =
\Acal \setminus \{a_r\}$ for all $r \in \{1,\ldots,n-1\}$.  Note that
$\Ccal_1$ is the projective monomial
curve associated to the arithmetic sequence $m_2 < \cdots < m_n$ and, thus, its Noether
resolution can also be  obtained by means of Theorem \ref{lipatil}.
Also when $n = 3$, 
$\Ccal_2$ is the curve associated to the arithmetic sequence $m_1 < m_3$.
For this reason, the rest of this section only concerns the study of
the multigraded Noether resolution of $\Ccal_r$ for $r \in \{2,\ldots,n-1\}$ and $n \geq 4$. 

\begin{remark}
Denote by $\Ccal_n$ and $\Ccal_{n+1}$ the Zariski closure of $\pi_{n}(\Ccal)$ and
$\pi_{n+1}(\Ccal)$ respectively. Then, both $\Ccal_n$ and $\Ccal_{n+1}$  are projective monomial
curves associated to arithmetic sequences and, thus, their Noether
resolutions can also be obtained by means of Theorem \ref{lipatil}. More
precisely, the corresponding arithmetic sequences are $m_1 < \cdots < m_{n-1}$ for $\Ccal_n$ and $1 <
2 < \cdots < n-1$ for $\Ccal_{n+1}$, i.e., $\Ccal_{n+1}$ is the
rational normal curve of degree $n-1$.
\end{remark}

We denote by $\Pcal_r$ the semigroup generated by $\Acal_r$ for $r \in \{2,\ldots,n-1\}$ and $n \geq 4$. Proposition 
\ref{Semigroup_QuasiArithmetics} shows how to get the semigroups $\Pcal_r$
from $\Scal$. In the proof of this result we will use the
following two lemmas, both of them can be directly deduced from
\cite[Lemma 2.1]{BerGarGar15}.

\begin{lemma}\label{min}
Set $q := \lfloor (m_1 - 1)/ (n-1) \rfloor \in \N$; then,
\begin{itemize}
\item[(a)] $q + d + 1 = {\rm min}\{b \in \Z^+\, \vert \, b m_1
\in \sum_{i = 2}^n \N m_i\}$
\item[(b)]  $q + 1 = {\rm min}\{b \in
\Z^+\, \vert \, b m_n \in \sum_{i = 1}^{n-1} \N m_i\}$
\item[(c)] $(q+d)a_1 + a_{i} = a_{l+i} + q a_n + d a_{n+1}$ for all $i \in \{1,\ldots,n-l\}$, where $l := m_1 - q(n-1) \in
\{1,\ldots,n-1\}$.
\end{itemize}
\end{lemma}

\begin{lemma}\label{smg} For all $r \in \{2,\ldots,n-1\}$, we have that $m_r \in \sum_{i \in \{1,\ldots,n\}
\setminus \{r\}} \N m_i$ if and only if $r > m_1$.
\end{lemma}

 \begin{proposition}\label{Semigroup_QuasiArithmetics} Set $q :=
\lfloor (m_1 - 1)/ (n-1) \rfloor$ and $l := m_1 - q(n-1)$. If $r\leq
m_1$, then
\begin{itemize}
\item[(a.1)] for $r=2$, \\ $\Scal \setminus \Pcal_2 = \left\{\begin{array}{ll} \left\{\mu a_1 + a_2 + \lambda\,a_{n+1}\ \vert\
0 \leq \mu \leq q+d-1,\lambda\in\N\right\}, &  {\rm if} \ l
\neq n-1, \\
 \left\{\mu a_1 + a_2 + \lambda\,a_{n+1}\ \vert\
0 \leq \mu \leq q+d,\lambda\in\N\right\}, & {\rm if}\ l =
n-1,\end{array}\right.$
\item[(a.2)] for $r \in
\{3,\ldots,n-2\}$, $\Scal\setminus \Pcal_r = \left\{a_r+\lambda\,a_{n+1}\ \vert\ \lambda\in\N\right\},$ and
\item[(a.3)] for $r=n-1$,  \\ $\Scal \setminus \Pcal_{n-1} = \left\{\begin{array}{ll} \left\{a_{n-1} + \mu a_n + \lambda\,a_{n+1}\ \vert\
0\leq\mu\leq q-1 {\rm \ or\ }0\leq\lambda\leq d-1
\right\}, & {\rm if}\ l \neq n-1, \\
\left\{a_{n-1} + \mu a_n + \lambda\,a_{n+1}\ \vert\
0\leq\mu\leq q {\rm \ or\ } 0\leq\lambda\leq d-1
\right\}, & {\rm if}\ l = n-1.\end{array}\right.$
\end{itemize} If $r > m_1$, then
\begin{itemize}
\item[(b.1)] for $r=2$, $\Scal\setminus \Pcal_2 = \{\mu a_1 + a_2+\lambda a_{n+1}\ \vert\
0 \leq \mu,\lambda \leq d-1\}$ ,
\item[(b.2)] for $r\in\{3,\ldots,n-2\}$, $\Scal\setminus \Pcal_r = \{a_r+\lambda a_{n+1}\ \vert\ 0 \leq \lambda \leq d-1\}$, and
\item[(b.3)] for $r=n-1$, $\Scal\setminus \Pcal_{n-1} = \{a_{n-1}+ \mu a_n + \lambda a_{n+1}\ \vert\ \mu \in \N, 0 \leq \lambda \leq d-1\}$.
\end{itemize}
\end{proposition}
\begin{proof} We express every $s \in \Scal$ as $s =
\alpha_1 a_1 + \epsilon_i a_i + \alpha_n a_n + \alpha_{n+1}
a_{n+1}$, with $\alpha_1, \alpha_n, \alpha_{n+1} \in \N$, $i \in
\{2,\ldots,n-1\}$ and $\epsilon_i \in \{0,1\}$. Whenever $\epsilon_i
= 0$ or $i \neq r$, it is clear that $s \in \Pcal_r$. Hence, we
 assume that $s = \alpha_1 a_1 + a_r + \alpha_n a_n + \alpha_{n+1}
a_{n+1}$ and the idea of the proof is to characterize the values of
$\alpha_1,\alpha_n,\alpha_{n+1}$ so that $s \in \Pcal_r$ in each
case.

Assume first that $r \in \{3,\ldots,n-2\}$ and let us prove {\it
(a.2)} and {\it (b.2)}. If $\alpha_1
> 0$ or $\alpha_n > 0$, the equalities $a_1 + a_r = a_2 + a_{r-1}$
and $a_r + a_n = a_{r+1} + a_{n-1}$ yield that $s \in \Pcal_r$, so
it suffices to consider when $s = a_r + \alpha_{n+1} a_{n+1}$. If $r
\leq m_1$, then by Lemma \ref{smg} we get that $s \notin \Pcal_r$
because the first coordinate of $s$ is precisely $m_r$. This proves
{\it (a.2)}. If $r > m_1$ and $\alpha_{n+1} \geq d$, then the
equality $a_r + d a_{n+1} = d a_1 + a_{r-m_1}$ yields that $s \in
\Pcal_r$. However, if $\alpha_{n+1} < d$ we are proving that $s
\notin \Pcal_r$. Suppose by contradiction that $s \in \Pcal_r$ and $\alpha_{n+1}<d$, then
\begin{equation}\label{eq1} s = a_r + \alpha_{n+1} a_{n+1} = \sum_{j \in \{1,\ldots,n+1\}
\setminus \{r\}} \beta_j a_j \end{equation} for some $\beta_j \in
\N$, then $d \geq 1 + \alpha_{n+1} = \sum_{j \in \{1,\ldots,n+1\}
\setminus \{r\}} \beta_j$. Moreover, observing the first coordinates
in (\ref{eq1}) we get that $m_r = \sum_{j \in \{1,\ldots,n\}
\setminus \{r\}} \beta_j m_j$. Hence, $m_1 + (r - 1)d = \sum_{j
\{1,\ldots,n\} \setminus \{r\}} \beta_j (m_1 + (j-1) d)$ and, since
$\gcd\{m_1,d\}=1$, this implies that $d$ divides $(\sum_{j
\{1,\ldots,n\} \setminus \{r\}} \beta_j) - 1$, but $0 < (\sum_{j \in
\{1,\ldots,n\} \setminus \{r\}} \beta_j) - 1 < d$, a contradiction.
Thus {\it (b.2)} is proved.

Since the proof of {\it (a.1)} is similar to the proof of {\it
(a.3)} we are not including it here. So let us prove {\it (b.1)}.
Assume that $r = 2$. If $\alpha_n > 0$ the equality $a_2 + a_n =
a_{3} + a_{n-1}$ yields that $s \in \Pcal_2$, so it suffices to
consider when $s = \alpha_1 a_1 + a_r + \alpha_{n+1} a_{n+1}$. If
$\alpha_1 \geq d$, then the identity $d a_1 + a_2 = a_3 + d a_{n+1}$
yields that $s \in \Pcal_2$. For $\alpha_1 < d$, if $\alpha_{n+1}
\geq d$, the equality $\alpha_1 a_1 + a_2 + d a_{n+1} = (\alpha_1 +
d + 1) a_1$ also yields that $s \in \Pcal_2$. Thus, to conclude {\it
(b.1)} it only remains to proof that $s \notin \Pcal_2$ when
$\alpha_1, \alpha_{n+1} < d$. Indeed, assume that $\alpha_1 a_1 +
a_2 + \alpha_{n+1} a_{n+1} = \sum_{j \in \{1,3,\ldots,n+1\}} \beta_j
a_j$. Observing the first coordinate of the equality we get that
$\alpha_1 + m_2 = \sum_{j \in \{1,3,\ldots,n\}} \beta_j m_j$, but
$\alpha_1 + m_2 < m_3 < \cdots < m_n$, so $\beta_3 = \cdots =
\beta_{n+1} = 0$. But this implies that $\beta_1 = \alpha_1 + d + 1$
and, hence, $\beta_{n+1} < 0$, a contradiction.

Assume now that $r = n-1$. If $\alpha_1 > 0$, the equality $a_1 +
a_{n-1} = a_{2} + a_{n-2}$ yields that $s \in \Pcal_{n-1}$, so it
suffices to consider when $s = a_{n-1} + \alpha_n a_n + \alpha_{n+1}
a_{n+1}$. Whenever $s \in \Pcal_{n-1}$, then $s$ can be expressed as
$s = \sum_{j \in \{1,\ldots,n-2,n,n+1\}} \beta_j a_j$, if we
consider both expressions of $s$, we get that
\begin{itemize}\item[(i)] $\sum_{j \in \{1,\ldots,n-2,n,n+1\}}
\beta_i = 1 + \alpha_n + \alpha_{n+1}$, and \item[(ii)] $\sum_{j \in
\{1,\ldots,n-2,n\}} \beta_j m_j = m_{n-1} + \alpha_n
m_n.$\end{itemize} If $\alpha_{n+1} < d$ we are proving that $s
\notin \Pcal_{n-1}$. Assume by contradiction that $s \in
\Pcal_{n-1}$. From (ii) and Lemma \ref{min} we deduce that $\beta_n
< \alpha_n$. Moreover,  if we expand (ii) considering that $m_i =
m_1 + (i-1)d$ for all $i \in \{1,\ldots,n\}$ and that $\gcd\{m_1,d\}
= 1$, we get that $d$ divides $\sum_{j \in \{1,\ldots,n-2,n\}}
\beta_j - \alpha_n - 1 = \alpha_{n+1} - \beta_{n+1}$, a
contradiction to $0 < \alpha_{n+1} - \beta_{n+1} < d$.

\noindent {\emph Case 1: $m_1 \geq n-1$}. Assume that $s \in
\Pcal_{n-1}$. By (ii) and Lemma \ref{smg} we have that $\beta_n <
\alpha_n$, so there exists $j_0 \in \{1,\ldots,n-2\}$ such that
$\beta_{j_0} > 0$. As a consequence, if we add $d - \beta_n m_n$ in
both sides of (ii) we get that $(\alpha_n + 1 - \beta_n)m_n =
\sum_{j \in \{1\ldots,n-2\}} \beta_j m_j - m_{j_0} + m_{j_0  + 1}
\in \sum_{j \in \{1,\ldots,n-1\}} \N m_j$. Hence, by Lemma \ref{min}
we have that $\alpha_n \geq \alpha_n - \beta_n \geq q$. If $l <
n-1$, for $\alpha_n \geq q$, $\alpha_{n+1} \geq d$ the equality of
Lemma \ref{min} $(q+d) a_1 + a_{n-l-1} = a_{n-1} + q a_n + d
a_{n+1}$ shows that $s \in \Pcal_{n-1}$. This proves {\it (a.3)}
whenever $l\leq n-1$. If $l = n-1$, for $\alpha_n \geq q+1$,
$\alpha_n \geq d$, again the equality $(q+d+1) a_1 = (q+1)a_n + d
a_{n+1}$ shows that $s \in \Pcal_{n-1}$. It only remains to prove
that if $\alpha_n = q$; then $s \notin \Pcal_2$. Assume by
contradiction that $a_{n-1} + q a_n + \alpha_{n+1} a_{n+1} = \sum_{j
\in \{1,\ldots,n-2,n,n+1\}} \beta_j a_j$. Then, the first
coordinates of this equality yield that $m_{n-1} + q m_n = \sum_{j
\in \{1,\ldots,n-2,n\}} \beta_j m_j$ and we deduce by Lemma
\ref{smg} that $\beta_n < q$ and, hence, there exists $j_0 \in
\{1,\ldots,n-2\}$ such that $\beta_{j_0}
> 0$. We denote $\beta_{n-1} := 0$, $\lambda_j := \beta_j$ for all $j
\in \{1,\ldots,n\} \setminus \{j_0, j_0 - 1\}$, $\lambda_{j_0} =
\beta_{j_0} - 1$, $\lambda_{j_0+1} = \beta_{j_0+1} + 1$, then adding
$d$ in both sides of the equality and using Lemma \ref{min}, we get
that $(q+1)m_n = (q+d+1)m_1 = \sum_{j \in \{1,\ldots,n\}} \lambda_j
m_j \in \sum_{i = 1}^n \N m_i$. However, $\lambda_1 \neq q+d+1$,
$\lambda_{n} \neq q+1$, so applying iteratively the equalities $a_i
+ a_j = a_{i-1} + a_{j+1}$ for all $2 \leq i \leq j \leq n-1$ we
express $\sum_{j \in \{1,\ldots,n\}} \lambda_j m_j$ as $\mu_1 m_1 +
\epsilon_k m_k + \mu_n m_n$ with $\mu_1, \mu_m \in \N$, $k \in
\{2,\ldots,n-1\}$, $\epsilon_k \in \{0,1\}$. It is clear that $\mu_1
\neq q+d+1$ and that $\mu_n \neq q+1$ and one of those is nonzero,
so this contradicts the minimality of $q+d+1$ or $q+1$.

To prove {\it (b.3)} it only remains to prove that if $\alpha_{n+1}
\geq d$, then $s \in \Pcal_{n-1}$, but this easily follows from the
relation $a_{n-1} + d a_{n+1} = d a_1 + a_{n-1-m_1}$.
\end{proof}

\medskip From the previous result and  Proposition \ref{S0} it is not difficult to obtain the following corollary,
which provides the shifts of the first step of a multigraded Noether resolution
of $K[\Pcal_r]$ for all $r \in \{2,\ldots,n-1\}$, namely
$(\Pcal_r)_0 := \{s \in \Pcal_r \, \vert \, s - a_n, s - a_{n+1}
\notin \Pcal_r\}$. Indeed, Corollary \ref{S_0_QuasiArithmetics} describes $(\Pcal_r)_0$ from the set $\Scal_0$ given
by Theorem \ref{lipatil}.

 \begin{corollary}\label{S_0_QuasiArithmetics}We denote $t_{\mu} :=
\mu a_1 + a_2$ for all $\mu \in \N$. If $r\leq m_1$, then
\begin{itemize}
\item[(a.1)] for $r=2$, \\ $(\Pcal_2)_0 = \left\{\begin{array}{ll} (\Scal_0  \setminus \left\{t_{\mu} \, \vert \,
0 \leq \mu \leq q+d-1 \right\}) \cup \left\{t_{\mu} + a_n\, \vert \,
 0 \leq \mu \leq q+d-1 \right\}, & {\rm if}\ l \neq n-1,\\
(\Scal_0  \setminus \left\{t_{\mu} \, \vert \,
0 \leq \mu \leq q+d \right\}) \cup \left\{t_{\mu} + a_n\, \vert \, 0
\leq \mu \leq q+d \right\}, & {\rm if} \ l = n-1,\end{array}\right.$
\item[(a.2)] for $r \in
\{3,\ldots,n-2\}$, $(\Pcal_r)_0 = (\Scal_0 \setminus \{a_r\}) \cup \{a_r+a_n\}$,
\item[(a.3)] for $r=n-1$, \\ $(\Pcal_{n-1})_0 = \left\{\begin{array}{ll} (\Scal_0 \setminus \{a_{n-1}\}) \cup \{a_{n-1}+ q a_n + d a_{n+1}\},
 & {\rm if}\  l\neq n-1,\\
(\Scal_0 \setminus \{a_{n-1}\}) \cup \{a_{n-1}+ (q+1) a_n + d a_{n+1}\}, & {\rm if} \ l = n-1.\end{array}\right.$
\end{itemize} If $r > m_1$, then
\begin{itemize}
\item[(b.1)] for $r=2$, $(\Pcal_2)_0 = (\Scal_0  \setminus \left\{t_{\mu} \, \vert \,
0 \leq \mu \leq d-1 \right\}) \cup \left\{t_{\mu} + a_n, t_{\mu} + d
a_{n+1}\, \vert \, 0 \leq \mu \leq d-1 \right\}$,
\item[(b.2)] for $r\in\{3,\ldots,n-2\}$, $(\Pcal_r)_0 = (\Scal_0  \setminus \left\{a_r \right\}) \cup \left\{a_r+a_n, a_r+d a_{n+1}\right\}$ , and
\item[(b.3)] for $r=n-1$, $(\Pcal_{n-1})_0 = (\Scal_0  \setminus \left\{a_{n-1} \right\}) \cup \left\{a_{n-1}+d a_{n+1}\right\}$.
\end{itemize}

\end{corollary}

\medskip
From Corollary \ref{S_0_QuasiArithmetics} and Proposition \ref{CMcharacSemigroup}, we get the 
following characterization of the Cohen-Macaulay property for this
family of semigroup rings taking into account that $D$ in Proposition \ref{CMcharacSemigroup} equals $m_n$ in these cases.

\begin{corollary}\label{CMcharacterization_QuasiArithmetics}
$K[\Pcal_r]$ is Cohen-Macaulay $\Longleftrightarrow$ $r \leq m_1$ or
$r = n-1$.
\end{corollary}

Moreover, as a consequence of Theorem
\ref{macaulayfication_semigroups} and Corollary
\ref{S_0_QuasiArithmetics}, we get the following result.

\begin{corollary}\label{macaulayfication} For all $r \in
\{2,\ldots,n-2\}$ and $r > m_1$, the 
Macaulayfication of $K[\Pcal_r]$ is $K[\Scal]$.
\end{corollary}

In order to get the whole multigraded Noether resolution of $K[\Pcal_r]$
for all $r \in \{2,\ldots,n-2\}$ and $r > m_1$, it remains to study
 its second step. By Theorem \ref{S_1}, its shifts are given by the set $(\Pcal_r)_1 := \{s \in
\Pcal_r \, \vert \, s - a_{n}, s - a_{n+1} \in \Pcal_r$ and $s - a_n -
a_{n+1} \notin \Pcal_r\}$.

 \begin{corollary}\label{S_1_QuasiArithmetics}
$$(\Pcal_r)_1 = \left\{\begin{array}{ll} \left\{ \mu a_1 + a_2 + a_n +
da_{n+1}\ \vert \mu\in\{0,\ldots,d-1\}\right\}, & {\rm if} \ r = 2 {\rm \ and \ } m_1=1, \\
\{a_r+a_n+da_{n+1}\}, & {\rm if}\ r \in
\{3,\ldots,n-2\} \ {\rm and} \  m_1 < r. \end{array} \right. $$
\end{corollary}

\medskip
As a consequence of the above results, we are able to provide the multigraded 
 Noether resolution of $K[\Pcal_r]$ for all $r \in
\{2,\ldots,n-1\}$.

\begin{theorem}\label{resolution_QuasiArithmetics}Let $q,l \in \N$ be the
integers $q := \lceil (m_1-1)/(n-1) \rceil$ and $l := m_1 - q(n-1)$.
If we set $s_{\lambda} :=
\left(\left\lceil\frac{\lambda}{n-1}\right\rceil m_n-\lambda
d,\,\lambda d \right) \in \N^2$ for all $\lambda \in
\{0,\ldots,m_n-1\}$, then the multigraded Noether resolution of $K[\Pcal_r]$ is
given by the following expressions: \begin{itemize}
\item For $m_1 \geq 2$, then
$$ 0 \xrightarrow{} \left( \oplus_{\lambda = 0,\, \lambda \notin
\Lambda_1}^{m_n-1} A \cdot s_{\lambda} \right)  \oplus \bigg(
\oplus_{\lambda \in \Lambda_1} A\cdot(s_{\lambda}+a_n) \bigg)
\xrightarrow{}  K[\Pcal_2] \xrightarrow{} 0,$$ where $\Lambda_1 :=
\{\mu(n-1) - 1\, \vert \, 1 \leq \mu \leq q+d+\epsilon\}$, and
$\epsilon = 1$ if $l = n-1$, or $\epsilon = 0$ otherwise.
\item For $r \in \{3,\ldots,n-2\}$ and $r \leq m_1$, then
$$0 \xrightarrow{} \left(\oplus_{\lambda=0,\,\lambda\neq n-r}^{m_n-1} A\cdot s_{\lambda}\right)  \oplus
A\cdot (a_{r}+a_n)\xrightarrow{}  K[\Pcal_r] \xrightarrow{} 0$$
\item For $r = n-1 \leq m_1$, then
$$0 \xrightarrow{} \left(\oplus_{\lambda=0,\, \lambda \neq 1}^{m_n-1} A\cdot s_{\lambda}\right)  \oplus
A\cdot(a_{n-1}+ (q + \epsilon)a_n + d a_{n+1}) \xrightarrow{}
K[\Pcal_{n-1}] \xrightarrow{} 0,$$ where $\epsilon = 1$ if $l =
n-1$, or $\epsilon = 0$ otherwise.
\item For $m_1 = 1$, then
$$ 0 \xrightarrow{} \oplus_{\lambda\in\Lambda_2} A\cdot(s_{\lambda}+a_n+da_{n+1})
\xrightarrow{} \begin{array}{c} \oplus_{\lambda=0,\lambda\notin\Lambda_2}^{m_n-1} A\cdot s_{\lambda} \\
 \oplus \\ \oplus_{\lambda\in\Lambda_2} A\cdot (s_{\lambda}+a_n) \\ \oplus \\  \oplus_{\lambda\in\Lambda_2} A
\cdot(s_{\lambda} + d a_{n+1})
\end{array} \xrightarrow{}  K[\Pcal_2] \xrightarrow{} 0,$$where
$\Lambda_2 := \{\mu(n-1) - 1\, \vert \, 1 \leq \mu \leq d\}$.

\item For $r \in \{3,\ldots,n-2\}$ and $r > m_1$, then
$$ 0 \xrightarrow{} A\cdot(a_{r}+ a_n + d a_{n+1})
\xrightarrow{} \begin{array}{c}
\left(\oplus_{\lambda=0,\,\lambda\neq
n-r}^{m_n-1} A\cdot s_{\lambda}\right) \\ \oplus \\
A\cdot(a_{r}+a_n) \oplus A\cdot(a_{r}+d a_{n+1})\end{array}
\xrightarrow{} K[\Pcal_r] \xrightarrow{} 0.$$
\item For $r = n-1 > m_1$, then
$$0 \xrightarrow{} \left(\oplus_{\lambda=0,\, \lambda \neq 1}^{m_n-1} A\cdot s_{\lambda}\right)  \oplus
A\cdot(a_{n-1}+ d a_{n+1})\xrightarrow{}  K[\Pcal_{n-1}]
\xrightarrow{} 0.$$
\end{itemize}
\end{theorem}

\medskip It is worth pointing out that from Theorem \ref{resolution_QuasiArithmetics} and Remark \ref{multtostandardgraded}, one 
can obtain the Noether resolution of $K[\Pcal_r]$ with respect to the standard grading. 
In addition, the description of $(\Pcal_r)_i$ for all $r \in \{2,\ldots,n-1\}$, $i \in \{0,1\}$,
allows us to use Remark \ref{multtostandardgraded} to
provide a formula for the Castelnuovo-Mumford regularity of
$K[\Pcal_r]$.
\begin{theorem}\label{regularity_QuasiArithmetics}The
Castelnuovo-Mumford regularity of $K[\Pcal_r]$ equals:

\smallskip
\begin{center}
{\rm reg}$(K[\Pcal_r]) = \left\{
\begin{array}{ll}
\lceil \frac{m_n - 1}{n-1} \rceil + 1, & $ if $r \in \{2,n-1\}$ and $r \leq m_1, \\
2d, &$ if $r = 2$ and $m_1 = 1$, and $\\
\lceil \frac{m_n - 1}{n-1} \rceil, &$ if $r \in
\{3,\ldots,n-2\},$ or $r = n-1$ and $m_1 < r \end{array}\right.$
\end{center}
\end{theorem}

Let us illustrate the results of this section with an example.

\begin{example}Consider the  projective monomial curve given parametrically by:
$$x_1=st^6,x_2=s^5 t^2, x_4=s^7,x_5=t^7.$$
We observe that the curve corresponds to $\Ccal_2$, where $\Ccal$ is
the curve associated to the arithmetic sequence $m_1 < \cdots < m_n$
with $m_1 = 1$, $d = 2$ and $n = 4$. Hence, by Theorem
\ref{resolution_QuasiArithmetics}, we get that the multigraded
Noether resolution of $K[\Pcal_2]$ is

$$ 0 \xrightarrow{} A\cdot(10,18) \oplus A\cdot(11,24) \xrightarrow{} \begin{array}{c} A \oplus A\cdot(1,6) \oplus
A\cdot(5,2) \oplus \\  A\cdot(2,12) \oplus A\cdot(6,8) \oplus
A\cdot(10,4) \oplus
\\ A\cdot(3,18) \oplus A\cdot(11,10) \oplus A\cdot(4,24)
\end{array} \xrightarrow{} K[\Pcal_2] \xrightarrow{} 0.$$
By Corollary \ref{multiHilbert}, we get that the multigraded Hilbert series of $K[\Pcal_2]$ is
$$ HS_{K[\Pcal_2]}(t_1,t_2) = \frac{1 + t_1 t_2^6 + t_1^2t_2^{12} + t_1^{4}t_2^{24} + t_1^{3}t_2^{18} + t_1^5 t_2^2  + t_1^{6}t_2^{8}+ t_1^{10}t_2^{4}- t_1^{10}t_2^{18} + t_1^{11}t_2^{10} - t_1^{11}t_2^{24}}{(1 - t_1^7) (1 - t_2^7)}.$$

Following Remark \ref{multtostandardgraded}, if we consider the standard grading on $R$, we get
the following Noether resolution of $K[\Pcal_2]$: 
$$ 0 \xrightarrow{} A(-4) \oplus A(-5) \xrightarrow{} \begin{array}{c} A \oplus A(-1)^2 \oplus
A(-2)^3 \\  A(-3)^2 \oplus A(-4)
\end{array} \xrightarrow{} K[\Pcal_2] \xrightarrow{} 0,$$
and the following expression for the Hilbert series of $K[\Pcal_2]$:
$$ HS_{K[\Pcal_2]}(t) = \frac{1 + 2t + 3t^2 + 2t^3 - t^5}{(1 - t)^2}.$$
We also have that ${\rm reg}(K[\Pcal_2]) =  4$.
\end{example}

\section*{Acknowledgements}
The authors want to thank the anonymous referees for their comments and  suggestions that we believe have helped to improve this manuscript. In particular,
Section 5 was included to answer a question made by the referees.

The first three authors were supported by the Ministerio de
Econom\'ia y Competitividad,  Spain (MTM2013-40775-P and MTM2016-78881-P).

\end{document}